\documentclass[11pt]{amsart}

\usepackage{amssymb,amsmath}
\usepackage{mathrsfs}

\large\normalsize
\newtheorem{theorem}{Theorem}
\newtheorem{lemma}[theorem]{Lemma}
\newtheorem{cor}[theorem]{Corollary}

\newtheorem{definition}[theorem]{Definition}

\newtheorem{remark}[theorem]{Remark}

\newcommand{\R}{\ensuremath{\mathbb{R}}}
\newcommand{\Z}{\ensuremath{\mathbb{Z}}}
\newcommand{\C}{\ensuremath{\mathbb{C}}}
\newcommand{\N}{\ensuremath{\mathbb{N}}}
\newcommand{\T}{\ensuremath{\mathbb{T}}}

\newcommand{\nTilde}[1]{\widetilde{#1}}                    
\newcommand{\nHat}[1]{\widehat{#1}}                        

\newcommand{\cldone}{{\rm C}_{\rm ld}^1}
\newcommand{\cone}{{\rm C}^1}

\DeclareMathOperator{\diag}{diag}
\DeclareMathOperator{\rank}{rank}
\DeclareMathOperator{\const}{const}
\DeclareMathOperator{\im}{Im}
\DeclareMathOperator{\Dom}{Dom}
\DeclareMathOperator{\Ran}{Ran}
\DeclareMathOperator{\Ker}{Ker}
\DeclareMathOperator{\adj}{adj}
\DeclareMathOperator{\trans}{T}
\DeclareMathOperator{\tr}{tr}

\numberwithin{equation}{section}
\numberwithin{theorem}{section}

\small\normalsize

\begin{document}

\title[Unified Weyl–-Titchmarsh Theory]{Unifying discrete and continuous Weyl–-Titchmarsh theory via a class of linear Hamiltonian systems on Sturmian time scales}

\author{Douglas R. Anderson} 
\address{Department of Mathematics and Computer Science, Concordia College, Moorhead, MN 56562 USA \\ visiting the School of Mathematics, the University of New South Wales, Sydney NSW 2052, Australia}
\email{andersod@cord.edu}


\begin{abstract} 
In this study, we are concerned with introducing Weyl-Titchmarsh theory for a class of dynamic linear Hamiltonian nabla systems over a half-line on Sturmian time scales. After developing fundamental properties of solutions and regular spectral problems, we introduce the corresponding maximal and minimal operators for the system. Matrix disks are constructed and proved to be nested and converge to a limiting set. Some precise relationships among the rank of the matrix radius of the limiting set, the number of linearly independent square summable solutions, and the defect indices of the minimal operator are established. Using the above results, a classification of singular dynamic linear Hamiltonian nabla systems is given in terms of the defect indices of the minimal operator, and several equivalent conditions on the cases of limit point and limit circle are obtained, respectively. These results unify and extend certain classic and recent results on the subject in the continuous and discrete cases, respectively, to Sturmian time scales.
\end{abstract}

\maketitle\thispagestyle{empty}



\section{Introduction}\label{secintro}

Ahlbrandt \cite{ahl} introduced the following class of linear discrete Hamiltonian equations
$$ \nabla x(t)=H_y(t,x(t),y(t-1)), \quad \nabla y(t)=-H_x(t,x(t),y(t-1)), $$
where $\nabla x(t)=x(t)-x(t-1)$, which yields a linear discrete Hamiltonian system \cite{ap} of the form
$$ \nabla x(t)=A(t)x(t)+B(t)u(t-1), \quad  \nabla u(t)=C(t)x(t)-A^*(t)u(t-1), $$
where $A$, $B$, and $C$ are $d\times d$ matrices, $A^*$ is the complex conjugate transpose of $A$, and $B$ and $C$ are Hermitian.
Shi \cite{shi2} shifted the points one unit right and extended the analysis to develop a Weyl-Titchmarsh theory for linear discrete Hamiltonian systems
\begin{equation}\label{shi1.5}
 J\Delta y(t)=(\lambda W(t)+P(t))R(y)(t), \quad t\in[0,\infty)\cap\Z,
\end{equation}
where $W$ and $P$ are $2d\times 2d$ complex Hermitian matrices with weight function $W(t)\ge 0$, which is one possible discrete version (see also Clark and Gesztesy \cite{cg2}) of the classic form studied by Atkinson \cite{fva}
\begin{equation}\label{ms1.2}
 Jy'(t)=(\lambda W(t)+P(t))y(t), \quad t\in[0,\infty).
\end{equation}
Shi used the partial right uniform shift operator $R(y)(t)=(y_1^{\trans}(t+1),y_2^{\trans}(t))^{\trans}$ with the vector $y(t)=(y_1^{\trans}(t),y_2^{\trans}(t))^{\trans}$ for $y_1,y_2\in\C^d$ and the canonical symplectic matrix $J=\left(\begin{smallmatrix} 0 & -I_d \\ I_d & 0 \end{smallmatrix}\right)$. With a view toward extending \eqref{shi1.5} to Sturmian time scales, we introduce the linear Hamiltonian nabla system on Sturmian time scales
\begin{equation}\label{maineq}
 Jy^\nabla (t)=(\lambda W(t)+P(t))\Upsilon y(t), \quad t\in[t_0,\infty)_\T:=[t_0,\infty)\cap\T, \quad J=\left(\begin{smallmatrix} 0 & -I_d \\ I_d & 0 \end{smallmatrix}\right),
\end{equation}
where $W$ and $P$ are $2d\times 2d$ complex Hermitian left-dense continuous matrix functions, $y:\T\rightarrow\C^{2d}$ is a nabla differentiable vector function, and we use the partial left-shift operator $\Upsilon$ defined by 
\begin{equation}\label{leftshift}
 \Upsilon y(t):=\begin{pmatrix} y_1(t) \\ y_2^\rho(t)\end{pmatrix} \quad\text{for}\quad y(t)=\begin{pmatrix} y_1(t) \\ y_2(t)\end{pmatrix},  \quad y_1,y_2:[\rho(t_0),\infty)_\T\rightarrow\C^d. 
\end{equation} 
Throughout this work we have the following key assumptions: following Atkinson \cite[Chapter 9]{fva}, the weight function $W$ satisfies the definiteness conditions
\begin{equation}\label{shiA1}
 \begin{cases}
   W(t)=\diag\{W_1(t),W_2(t)\}\ge 0, & W_j(t)\ge 0 \; \text{is}\; d\times d \;\text{Hermitian}, \quad j=1,2, \\
   \displaystyle\int_{t_0}^t (\Upsilon y)^*(s)W(s)\Upsilon y(s)\nabla s > 0, & \forall\; t\in[t_1,\infty)_\T \; \text{for some}\; t_1\in\T, 
 \end{cases}
\end{equation}
for every nontrivial solution $y$ of \eqref{maineq}, and $P$ satisfies the block form
\begin{equation}\label{shiA2}
 P(t)=\begin{pmatrix} -C(t) & A^*(t) \\ A(t) & B(t) \end{pmatrix}, \quad I_d-\nu(t) A(t)\;\text{invertible}, \quad t\in[t_0,\infty)_\T, 
\end{equation}
where $A$, $B$, and $C$ are left-dense continuous $d\times d$ complex matrices with $B$ and $C$ Hermitian. Using standard notation, $\T$ is an unbounded Sturmian time scale; the left jump operator $\rho$ is given by $\rho(t)=\sup\{s\in\T:s<t\}$ with the composition $u\circ\rho$ denoted $u^\rho$; the graininess function is defined by $\nu(t)=t-\rho(t)$; and the nabla derivative of $x$ at $t\in\T$, denoted $x^\nabla(t)$, is the vector (provided it exists) given by
$$ x^\nabla(t):=\lim_{s\rightarrow t}\frac{x^\rho(t)-x(s)}{\rho(t)-s}. $$
Sturmian time scales, introduced in \cite{abv}, are a specialized class of time scales (closed, nonempty sets of real numbers) with the property that
\begin{equation}\label{sturmfact}
 \sigma(\rho(t))=\rho(\sigma(t)) \quad\text{for all}\quad t\in[t_0,\infty)_\T. 
\end{equation}
This crucial assumption allows us to identify a partial right-shift operator $\Upsilon^{-1}$ in terms of \eqref{leftshift}, namely
\begin{equation}\label{rightshift}
 \Upsilon^{-1} y(t):=\begin{pmatrix} y_1(t) \\ y_2^\sigma(t)\end{pmatrix} \quad\text{for}\quad y(t)=\begin{pmatrix} y_1(t) \\ y_2(t)\end{pmatrix},  \quad y_1,y_2:[\rho(t_0),\infty)_\T\rightarrow\C^d;
\end{equation}
note that on our Sturmian time scale $\T$, we have $\Upsilon(\Upsilon^{-1}y)(t)=\Upsilon^{-1}(\Upsilon y)(t)=y(t)$ for all $t\in[t_0,\infty)_\T$.
For more on general time scales using the nabla derivative, see \cite{aticig} and \cite[Chapter 3]{bp2}. 


\begin{remark}
We employ the nabla version in \eqref{maineq} for two reasons. One, it reflects the original form introduced by Ahlbrandt \cite{ahl}, and two, it contains the following two important dynamic models \cite{and2,agh}. The first is the linear Hamiltonian nabla system on general time scales 
\begin{equation}\label{nabham}
 x^\nabla(t) = A(t)x(t) + B(t)u^\rho(t), \quad
 u^\nabla(t) = \left[C(t)-\lambda\omega(t)\right]x(t) - A^*(t)u^\rho(t) 
\end{equation}
for $t\in[a,b]_\T$, where $A$, $B$, $C$ and $\omega$ are $d\times d$ matrices, $B$ and $C$ are Hermitian, $\omega>0$ is positive definite. It is straightforward to write \eqref{nabham} in the form \eqref{maineq}, by taking 
$$  J = \left(\begin{smallmatrix} 0 & -I_d \\ I_d & 0 \end{smallmatrix}\right), \quad y(t)=\left(\begin{smallmatrix} x(t)\\ u(t)\end{smallmatrix}\right), \quad P(t) = \left(\begin{smallmatrix} -C(t) & A^*(t) \\ A(t) & B(t) \end{smallmatrix}\right), \quad 
   W(t) = \left(\begin{smallmatrix} \omega(t) & 0 \\ 0 & 0 \end{smallmatrix}\right). $$
This model includes the second-order self-adjoint matrix equation \cite{and,ab}
$$ -(P_0X^\Delta)^\nabla(t)+Q(t)X(t)=0 $$
for Hermitian $P_0$ and $Q$ with $P_0$ invertible, by taking $A=0$, $B=(P_0^\rho)^{-1}$, $C=Q$, and $\lambda=0$ in \eqref{nabham}.

The second important dynamic model is the even-order self-adjoint Sturm-Liouville dynamic equation
\begin{eqnarray} \label{My}
My(t) &=& \sum\limits_{k=0}^n (-1)^{n-k}\left( p_{n-k}y^{\nabla^{n-k-1}\Delta} \right)^{\Delta^{n-k-1}\nabla}(t) \\ \nonumber
  &=& (-1)^n\left(p_ny^{\nabla^{n-1}\Delta}\right)^{\Delta^{n-1}\nabla}(t) + 
  \dots  - \left(p_{3}y^{\Delta^2\nabla}\right)^{\nabla^2\Delta}(t) \\ 
  & & + \left(p_{2}y^{\nabla\Delta}\right)^{\Delta\nabla}(t) - \left(p_{1}y^\Delta\right)^\nabla(t) 
      + p_0(t)y(t), \nonumber
\end{eqnarray}
which is formally self-adjoint \cite{agh}, where $p_n\neq 0$. We will show \eqref{My} can be written in the form of \eqref{nabham}, where
\begin{gather}
  A=(a_{ij})_{1\leq i,j\leq n} \quad\text{with}\quad
   a_{ij}=\begin{cases}1: & \text{if } j=i+1,\; 1\leq i\leq n-1, \\
                       0: & \text{otherwise,} \end{cases} \label{ABCSL} \\
  B=\diag\left\{0,\dots,0,\frac{1}{p_n^\rho}\right\},\quad
  C=\diag\left\{p_0,p_1^\rho,p_2^\rho,\ldots,p_{n-1}^\rho\right\}. \nonumber
\end{gather}
\noindent To do this, we introduce the pseudo-derivatives of the function $y$ given by
\begin{eqnarray*}
  y^{[k]} &=& y^{\nabla^k}, \quad 0\le k\le n-1, \quad y^{[0]} = y^{\nabla^0} = y,   \\
  y^{[n]} &=& p_ny^{\nabla^{n-1}\Delta},  \\
  y^{[n+k]} &=& p_{n-k}y^{\nabla^{n-k-1}\Delta} - \left(y^{[n+k-1]}\right)^\Delta  \\ 
   &=&\sum_{i=0}^k(-1)^{k-i}\left(p_{k-i}y^{\nabla^{n-i-1}\Delta}\right)^{\Delta^{k-i}}, \quad 1 \le k \le n-1,  \\
  y^{[2n]} &=& p_0y - \left(y^{[2n-1]}\right)^\nabla=My.
\end{eqnarray*}
Then using the substitution
$$ x = \left(\begin{smallmatrix} y^{[0]} \\ y^{[1]} \\ \vdots \\ y^{[n-1]} \end{smallmatrix}\right), \quad
   u = \left(\begin{smallmatrix} y^{[2n-1]} \\ y^{[2n-2]} \\ \vdots \\ y^{[n]} \end{smallmatrix}\right), $$
and the matrices $A$, $B$, and $C$ above in \eqref{ABCSL}, we have that 
$$ x^\nabla = A(t)x + B(t)u^\rho, \quad u^\nabla = C(t)x - A^*(t)u^\rho, $$
and the example is complete.
\end{remark}

There is a vast literature on continuous linear Hamiltonian systems, and of late a growing collection of results on corresponding discrete Hamiltonian systems. For a few of the many relevant papers in the continuous case, see \cite{bgms},\cite{cg3}$-$\cite{cfhl},\cite{Reid3},\cite{az}, and see \cite{ahl,ap,bdk,cg2,jy1,jy2},\cite{shi1}$-$\cite{ws} for recent work on discrete second-order difference equations and linear Hamiltonian systems. Some of the fundamental continuous and disrete results alluded to above have been unified and extended via dynamic equations on time scales, introduced by Hilger \cite{Hilger}. For scalar Sturm-Liouville results on time scales, see \cite{abw,abv,kong,messer}, and for systems see \cite{AlBoRi,bp1,Hilscher}. Turning to Weyl \cite{weyl} and Titchmarsh \cite{titch} specifically, much has been published on the continuous Weyl-Titchmarsh theory, for example \cite{ehs,ek1,ek2},\cite{hsc1,hsc2},\cite{hsh1,hsh2,hsh3}, and three substantial works on the corresponding discrete theory, Atkinson \cite{fva}, Clark and Gesztesy \cite{cg2} and Shi \cite{shi2}. Recently \cite{weiss} made a first start on the scalar theory on time scales. There is yet to be any published work on a unified continuous and discrete Weyl-Titchmarsh theory, however, for linear Hamiltonian systems; via this paper we hope to initiate such an investigation. With Sturmian time scales there is a much broader scope of discretization options other than just the uniform step size offered by difference equations.  In the analysis that follows, we will largely follow a development of the theory along the lines of Shi \cite{shi2}.

We will proceed as follows. In Section 2, we introduce fundamental properties for system \eqref{maineq} and introduce a Lagrange identity. Regular spectral problems are discussed in Section 3 via separated boundary conditions, and a result on eigenpairs is given. In Section 4 we introduce a weighted Hilbert space to facilitate a study of maximal and minimal operators. Weyl disks and their limiting set are the focus of Section 5, while in Section 6 we introduce the concept of square summable solutions. In Section 7, we give a classification of singular linear Hamiltonian nabla systems. In the final section we discuss an alternative form to \eqref{maineq} that may also serve as a generalization of \eqref{shi1.5} and \eqref{ms1.2} on Sturmian time scales.


\section{Fundamental Properties}\label{funprop}

For any given $\lambda\in\C$, using the assumptions on the block forms of $W$ and $P$ in \eqref{shiA1} and \eqref{shiA2}, respectively, we can rewrite \eqref{maineq} as the pair of $d$-vector equations
\begin{equation}\label{shi2.1}
 \begin{cases} y_1^\nabla(t) = A(t)y_1(t)+\big(B(t)+\lambda  W_2(t)\big)y_2^\rho(t), \\
 y_2^\nabla(t) = \big(C(t)-\lambda  W_1(t)\big)y_1(t)-A^*(t)y_2^\rho(t). \end{cases}
\end{equation}
By \eqref{shiA2}, we have that
\begin{equation}\label{Atilde}
  E(t):=\big(I_d-\nu(t) A(t)\big)^{-1} 
\end{equation}
exists. Then we may also view solutions $y=(y_1,y_2)^{\trans}$ of \eqref{maineq} and \eqref{shi2.1} as solutions of 
\begin{equation}\label{nabsym}
 y^\nabla(t)=\mathcal{S}(t,\lambda)y(t), \quad \mathcal{S}(\cdot,\lambda):=\begin{pmatrix}  A-\nu\big(B+\lambda W_2\big)E^*\big(C-\lambda W_1\big) & \big(B+\lambda W_2\big)E^* \\ E^*\big(C - \lambda W_1\big) & -E^{*}A^* \end{pmatrix}
\end{equation}
for $E$ in \eqref{Atilde}, where it is straightforward to check that $\mathcal{S}(\cdot,\lambda)$ satisfies
\begin{equation}\label{Stilde}  
 \mathcal{S}^*\left(\cdot,\overline{\lambda}\right)J + J\mathcal{S}(\cdot,\lambda) 
    =\nu\mathcal{S}^*\left(\cdot,\overline{\lambda}\right)J\mathcal{S}(\cdot,\lambda)
\end{equation}
for $t\in\T$. Directly from \eqref{Stilde} we have that
\begin{equation}\label{imsstar}
 \left(I_{2d}-\nu(t)\mathcal{S}(t,\lambda)\right)^*J \left(I_{2d}-\nu(t)\mathcal{S}(t,\lambda)\right) = J, 
\end{equation}
so that $I_{2d}-\nu(t)\mathcal{S}(t,\lambda)$ is invertible and thus $\mathcal{S}(\cdot,\lambda)$ is $\nu-$regressive. Given the results above, we now have the following lemma.


\begin{lemma}
Assume $I_d-\nu A$ is invertible on $\T$, and $\lambda\in\C$ is arbitrary. Then for any vector solution $y(\cdot,\lambda)$ of \eqref{maineq}$_{\lambda}$ and for any vector solution $z\left(\cdot,\overline{\lambda}\right)$ of \eqref{maineq}$_{\overline{\lambda}}$ we have
\begin{equation}\label{shi2.6}
 z^*\left(t,\overline{\lambda}\right)Jy(t,\lambda)=\const.
\end{equation}
\end{lemma}

\begin{proof}
Let $t\in\T$. By using the simple useful formula $y^\rho=y-\nu y^\nabla$ and \eqref{nabsym} we have that
$$ y^\rho(t,\lambda) = y(t,\lambda)-\nu(t)y^\nabla(t,\lambda) = \left(I_d-\nu(t)\mathcal{S}(t,\lambda)\right)y(t,\lambda). $$
From the nabla product rule we subsequently obtain
\begin{eqnarray*}
 \left(z^{*}\left(t,\overline{\lambda}\right)Jy(t,\lambda)\right)^\nabla 
 &=& z^{*\nabla}\left(t,\overline{\lambda}\right)Jy^\rho(t,\lambda) + z^*\left(t,\overline{\lambda}\right)Jy^{\nabla}(t,\lambda) \\
 &=& z^*\left(t,\overline{\lambda}\right)\mathcal{S}^*\left(t,\overline{\lambda}\right)J\left(I_d-\nu(t)\mathcal{S}(t,\lambda)\right)y(t,\lambda) 
     + z^*\left(t,\overline{\lambda}\right)J\mathcal{S}(t,\lambda)y(t,\lambda) \\
 &=& z^*\left(t,\overline{\lambda}\right) \left[\mathcal{S}^*\left(t,\overline{\lambda}\right)J + J\mathcal{S}(t,\lambda)
     -\nu(t)\mathcal{S}^*\left(t,\overline{\lambda}\right)J\mathcal{S}(t,\lambda)\right] y(t,\lambda) \\
 &=& 0,
\end{eqnarray*}
where the last line follows from \eqref{Stilde}.
\end{proof}

We now define a natural dynamic nabla differential operator for \eqref{maineq} via
\begin{equation}\label{shieq15}
 \mathscr{L}y(t):= J y^\nabla(t) - P(t)\Upsilon y(t), \quad y\in\cldone\left([\rho(t_0),b]_\T,\C^{2d}\right),
\end{equation}
where $\cldone$ is the space of $2d$-vector functions with left-dense continuous nabla derivatives on the given time scale interval. Then we have the following key result.


\begin{theorem}[Lagrange Identity]\label{Lagrange}
For all $x,y\in\cldone\left([\rho(t_0),b]_\T,\C^{2d}\right)$, where $x=\left(\begin{smallmatrix} x_1 \\ x_2 \end{smallmatrix}\right)$ and $y=\left(\begin{smallmatrix} y_1 \\ y_2 \end{smallmatrix}\right)$, we have
$$ \int_{\rho(t_0)}^{b} \big\{(\Upsilon x)^*\mathscr{L}y - (\mathscr{L}x)^*\Upsilon y \big\}(t)\nabla t = x^*(t)Jy(t)\Big|_{\rho(t_0)}^{b}. $$
\end{theorem}

\begin{proof}
Suppressing the variable $t$, we have
$$ (\Upsilon x)^*\mathscr{L}y = -x_1^*y_2^\nabla+x_1^*Cy_1-x_1^*A^*y_2^\rho+x_2^{\rho*}y_1^\nabla-x_2^{\rho*}Ay_1-x_2^{\rho*}By_2^\rho, $$
$$ (\mathscr{L}x)^*\Upsilon y = -x_2^{\nabla*}y_1+x_1^*Cy_1-x_2^{\rho*}Ay_1+x_1^{\nabla*}y_2^\rho-x_1^*A^*y_2^\rho-x_2^{\rho*}By_2^\rho,  $$
so that when we subtract the second from the first, we obtain
\begin{eqnarray*}
 (\Upsilon x)^*\mathscr{L}y - (\mathscr{L}x)^*\Upsilon y 
 &=& -x_1^*y_2^\nabla + x_2^{\rho*}y_1^\nabla + x_2^{\nabla*}y_1 -x_1^{\nabla*}y_2^\rho \\
 &=& -(x_1^*y_2)^\nabla+(x_2^*y_1)^\nabla = \left(x^*Jy\right)^\nabla(t).
\end{eqnarray*}
The result follows from the fundamental theorem of calculus.
\end{proof}

\begin{lemma}\label{lemma2.3}
Assume $I_d-\nu A$ is invertible on $\T$. For all $\lambda,\eta\in\C$, let $y(\cdot,\lambda)$ and $z(\cdot,\eta)$ be any solutions of \eqref{maineq}$_\lambda$ and \eqref{maineq}$_\eta$ , respectively. Then for any $t\in(t_0,\infty)_\T$, we have
$$ \left(\eta-\overline{\lambda}\right)\int_{\rho(t_0)}^{t} (\Upsilon y)^*(s,\lambda)W(s)\Upsilon z(s,\eta)\nabla s = y^*(t,\lambda)Jz(t,\eta)\Big|_{\rho(t_0)}^{t}. $$
\end{lemma}

\begin{proof}
By \eqref{shieq15}, $\mathscr{L}y(\cdot,\lambda)=\lambda W\Upsilon y(\cdot,\lambda)$ and $\mathscr{L}z(\cdot,\eta)=\eta W\Upsilon z(\cdot,\eta)$. Use of Theorem \ref{Lagrange} yields the result. 
\end{proof}


\section{Regular Spectral Problems}

In this section we analyze \eqref{maineq} over the time scale interval $[\rho(t_0),b]_\T$ of finite measure, obtaining some fundamental spectral results. To this end, consider the regular spectral problem for system \eqref{maineq} on $[\rho(t_0),b]_\T$ with the separated (homogeneous Dirichlet) boundary conditions
\begin{equation}\label{shi2.11}
 \alpha y^\rho(t_0)=0, \quad \beta y(b)=0, \qquad y\in\cldone\left([\rho(t_0),b]_\T,\C^{2d}\right),
\end{equation}
where $\alpha$ and $\beta$ are (normalized) $d\times 2d$ matrices that satisfy the following self-adjoint boundary conditions
\begin{eqnarray}
 \rank \alpha=d, & \alpha\alpha^*=I_d, & \alpha J \alpha^*=0, \label{shi2.12} \\
 \rank \beta=d, & \beta\beta^*=I_d, & \beta J \beta^*=0. \label{shi2.13}
\end{eqnarray}
We call these boundary conditions self-adjoint as they cause the Lagrange identity in Theorem \ref{Lagrange} to equal zero for $y\in\cldone\left([\rho(t_0),b]_\T,\C^{2d}\right)$.


\begin{lemma}\label{lemma3.1}
Let $\alpha$ and $\beta$ satisfy \eqref{shi2.12} and \eqref{shi2.13}, respectively. Then $y\in\cldone\left([\rho(t_0),b]_\T,\C^{2d}\right)$ satisfies \eqref{shi2.11} if and only if there exists a unique vector $\xi\in\C^{2d}$ such that
\begin{equation}\label{shi2.14}
 y^\rho(t_0)=M\xi, \quad y(b)=N\xi,
\end{equation}
where $M=(-J\alpha^*,0)$ and $N=(0,J\beta^*)$. Additionally, 
\begin{equation}\label{shi2.15}
 M^*JM=0=N^*JN, \quad \rank\left(\begin{smallmatrix} M \\ N \end{smallmatrix}\right)=2d.
\end{equation}
\end{lemma}

\begin{proof}
For \eqref{shi2.14}, see the continuous case in Kratz \cite[Proposition 2.1.1]{kratz} or Zettl \cite[Theorem 10.4.3]{az}, or in the discrete case in Shi \cite[Lemma 2.1]{shi1}; the time-scales case is unchanged \cite[Lemma 2.4]{and}. Equation \eqref{shi2.15} follows immediately from \eqref{shi2.12} and \eqref{shi2.13}, respectively.
\end{proof}

Assume that \eqref{shiA2} holds on $[\rho(t_0),b]_\T$. For each $\lambda\in\C$, let $\Phi(\cdot,\lambda)$ be a fundamental matrix solution for \eqref{maineq}$_{\lambda}$. It follows from \eqref{nabsym} that we can view $\Phi(\cdot,\lambda)$ as the solution of the initial value problem  
$$ \Phi^\nabla=\mathcal{S}(\cdot,\lambda)\Phi, \qquad \Phi^\rho(t_0)=I_{2d}, $$
and that a general solution of \eqref{maineq}$_{\lambda}$ can be written as
\begin{equation}\label{shi2.16}
 y(t,\lambda)=\Phi(t,\lambda)y^\rho(t_0,\lambda).
\end{equation}


\begin{theorem}
Assume \eqref{shiA1} and \eqref{shiA2}, and let $\alpha$ and $\beta$ satisfy \eqref{shi2.12} and \eqref{shi2.13}, respectively. Then for each $b\in[t_1,\infty)_\T$ $($where $t_1$ is specified in \eqref{shiA1}$)$, $\lambda$ is an eigenvalue of the boundary value problem \eqref{maineq}, \eqref{shi2.11} if and only if
\begin{equation}\label{shi2.18}
  \det\big(\Phi(b,\lambda)M-N\big)=0
\end{equation}
for the fundamental matrix solution $\Phi$ in \eqref{shi2.16}. Moreover, all the eigenvalues of \eqref{maineq} and \eqref{shi2.11} are real and can be numbered serially as in
\begin{equation}\label{shi2.19}
 |\lambda_0(b)| \le |\lambda_1(b)| \le |\lambda_2(b)| \le\cdots,
\end{equation}
such that the corresponding eigenfunctions $y(\cdot,\lambda_j(b))$ satisfy the orthonormality relation
\begin{equation}\label{ipb}
 \Big\langle y(\cdot,\lambda_j(b)), y(\cdot,\lambda_k(b)) \Big\rangle_b:=\int_{\rho(t_0)}^b (\Upsilon y)^*(t,\lambda_k(b))W(t)\Upsilon y(t,\lambda_j(b))\nabla t = \delta_{jk}.
\end{equation}
\end{theorem}

\begin{proof}
To show \eqref{shi2.18}, recall that $\lambda$ is an eigenvalue for the boundary value problem \eqref{maineq}, \eqref{shi2.11} with nontrivial eigenfunction $y(\cdot,\lambda)$ if and only if there exists a vector $\xi\in\C^{2d}$, $\xi\neq 0$, such that $y^\rho(t_0,\lambda)=M\xi$ and $y(b,\lambda)=N\xi$ by Lemma \ref{lemma3.1} and \eqref{shi2.14}, if and only if $\left(\Phi(b,\lambda)M-N\right)\xi=0$ by \eqref{shi2.16}. That the eigenvalues are real follows from Theorem \ref{Lagrange} and Lemma \ref{lemma3.1}. As in Atkinson \cite[Theorem 9.2.1]{fva}, we note that $\Phi(b,\lambda)$ consists of entire functions of $\lambda$, so that the left-hand side of \eqref{shi2.18} is also an entire function. Thus it has no complex zeros, and its zeros have no finite limit point, whence we have \eqref{shi2.19}, which may be a finite or infinite list and includes possible multiplicities.

If $\lambda_j(b)\neq \lambda_k(b)$, then the corresponding eigenfunctions $y(\cdot,\lambda_j(b))$ and $y(\cdot,\lambda_k(b))$ are orthogonal by Theorem \ref{Lagrange} and Lemma \ref{lemma3.1}. These functions can be normalized by taking $y(t,\lambda_j(b))/\|y(\cdot,\lambda_j(b))\|_b$, where
$$ \|x\|_b:=\big(\langle x,x\rangle_b\big)^{1/2} $$
in terms of \eqref{ipb}. Following Atkinson \cite[Section 9.3]{fva} and Shi \cite[Theorem 2.3]{shi2}, suppose $\lambda_j(b)$ is an eigenvalue with multiplicity $d_j$. Set
\begin{equation}\label{fva938}
 V_j=\left\{\xi\in\C^{2d}:\left(\Phi(b,\lambda_j(b))M-N\right)\xi=0 \right\}. 
\end{equation}
It follows that $V_j$ is a subspace of $\C^{2d}$ with $\dim V_j=d_j$, and $\lambda_j(b)$ appears exactly $d_j$ times in \eqref{shi2.19}, say
$$ \lambda_j(b), \quad  j=j'+1,\cdots,j'+d_j. $$
We will choose a basis $\left\{\xi_j\right\}_{j=j'+1}^{j'+d_j}$ of the set $V_j$ in \eqref{fva938} such that the corresponding eigenfunctions
$$ y(t,\lambda_j(b))=\Phi(t,\lambda_j(b))M\xi_j, \quad j=j'+1,\cdots,j'+d_j, $$
are mutually orthonormal. We apply a process of orthogonalization ala Atkinson \cite[9.3.13]{fva}. If we write
$$ u_j(b)=y(\rho(t_0),\lambda_j(b))=M\xi_j, \quad y(t,\lambda_j(b))=\Phi(t,\lambda_j(b))u_j(b), $$
then \eqref{ipb} is equivalent to 
\begin{equation}\label{shi2.21}
 u^*_r(b)K(b,\lambda_j(b))u_s(b)=\delta_{rs}, \quad j'+1\le r,s\le j'+d_j,
\end{equation}
where
$$ K(t,\lambda):=\int_{\rho(t_0)}^{t} (\Upsilon\Phi)^*(\tau,\lambda)W(\tau)\Upsilon\Phi(\tau,\lambda)\nabla\tau,  $$
and $\Upsilon\Phi(t,\lambda)$ denotes the partial left-shift operator $\Upsilon$ from \eqref{leftshift} acting on the last $d$ rows of the fundamental matrix $\Phi(t,\lambda)$ with respect to the variable $t$. Using the definiteness condition in \eqref{shiA1} and the invertibility of $\Phi(\cdot,\lambda_j(b))$, we see that $K(b,\lambda_j(b))>0$. On the other hand, the space $\nTilde{V}_j=\{u:u=M\xi\}$ has the same dimension $d_j$ as $V_j$, as $M\xi=N\xi=0$ always implies $\xi=0$ by \eqref{shi2.15}. By the invertibility of $K(b,\lambda_j(b))$, the space
$$ \nHat{V}_j=\{v:v=(K(b,\lambda_j(b)))^{1/2}u,\; u\in\nTilde{V}_j\} $$
has dimension $d_j$ as well, and an orthonormal basis $\{v_r\}_{r=j'+1}^{j'+d_j}$, that is 
$$ v_r^*v_s=\delta_{rs}, \quad j'+1 \le r,s \le j'+d_j. $$
From this we recover a basis for $\nTilde{V}_j$, namely
$$ u_r=(K(b,\lambda_j(b)))^{-1/2}v_r, \quad j'+1 \le r,s \le j'+d_j, $$
that satisfies \eqref{shi2.21}, and \eqref{ipb} follows.
\end{proof}


\section{Maximal and Minimal Operators}

In this section we introduce a weighted Hilbert space $L^2_W\left([\rho(t_0),\infty)_\T\right)$, define minimal and maximal operators corresponding to system \eqref{maineq}, and show that the minimal operator is symmetric, the maximal operator is densely defined, and the adjoint of the minimal operator is precisely the maximal operator. To clarify the notation to follow, we will denote the domain, range, and kernel of an operator $K$ by $\Dom(K)$, $\Ran(K)$, and $\Ker(K)$, respectively. Some concepts for linear operators in Hilbert spaces are first introduced; see \cite{weid}.

\begin{definition}\label{shidef2.1}
Let $X$ be a Hilbert space with inner product $\langle\cdot,\cdot\rangle$, and let $K:\Dom(K)\subset X\rightarrow X$
be a linear operator.
\begin{enumerate}
 \item $K$ is said to be densely defined if $\Dom(K)$ is dense.
 \item $K$ is said to be Hermitian if it is formally self-adjoint, i.e., 	$\langle Kf, g\rangle = \langle f, Kg\rangle$ for all $f,g\in\Dom(K)$.
 \item $K$ is said to be symmetric if it is Hermitian and densely defined.
 \item Let $K$ be a densely defined linear operator. The adjoint operator $K^*$ of $K$ is defined as $\Dom(K^*)=\{g\in X: \text{the functional}\; f\mapsto\langle g, Kf\rangle\; \text{is continuous on}\; \Dom(K)\}$ and $\langle K^*g, f\rangle = \langle g, Kf\rangle$ for all $f\in\Dom(K)$ and $g \in\Dom(K^*)$.
 \item For given $\lambda\in\C$, the subspace $\Ran\left(\overline{\lambda}-K\right)^\perp$ is called the defect space of $K$ and $\lambda$, and
$d(\lambda) = \dim \Ran\left(\overline{\lambda}-K\right)^\perp$ is called the defect index of $K$ and $\lambda$.
\end{enumerate}
\end{definition}

If $K$ is Hermitian, then $d(\lambda)$ is constant in the upper and lower half planes, respectively. Denote $d_+=d(i)$ and $d_-=d(-i)$. Then $d_+$ and $d_-$ are called the positive and negative defect indices of $K$, respectively. Further, if $K$ is densely defined, then $\overline{\Ran K}\oplus\Ker K^*= X$. Hence, if $K$ is symmetric, then $\Ran\left(\overline{\lambda}-K\right)^\perp=\Ker(\lambda-K^*)$. It follows that, for the symmetric operator $K$, $\Ran(-i-K)^\perp=\Ker(i-K^*)$ and $\Ran(i-K)^\perp=\Ker(i+K^*)$, whereby $d_+=\dim\Ker(i-K^*)$ and $d_-=\dim\Ker(i+K^*)$.
 
We now introduce the following linear spaces. On the time scale half line, let
$$ L^1\left([\rho(t_0),\infty)_\T\right):=\left\{y:[\rho(t_0),\infty)_\T\rightarrow\C^{2d}:\; y \;\text{is integrable on}\; [\rho(t_0),\infty)_\T \right\} $$
and let
$$ L^2_W\left([\rho(t_0),\infty)_\T\right):=\left\{y\in L^1\left([\rho(t_0),\infty)_\T\right): \int_{\rho(t_0)}^{\infty} \big((\Upsilon y)^*W\Upsilon y \big)(t)\nabla t <\infty \right\} $$
for the partial left-shift operator $\Upsilon$ given in \eqref{leftshift}, with inner product given by
\begin{equation}\label{ipinfty}
 \langle y,z \rangle:=\int_{\rho(t_0)}^{\infty} \big((\Upsilon z)^*W\Upsilon y \big)(t)\nabla t, 
\end{equation}
where the weight function $W$ is the $2d\times 2d$ nonnegative Hermitian left-dense continuous matrix satisfying \eqref{shiA1}. In a similar manner, define on the finite-length interval the linear space
\begin{equation}\label{linspace1}
  L^1\left([\rho(t_0),b]_\T\right):=\left\{y:[\rho(t_0),b]_\T\rightarrow\C^{2d}:\; y \;\text{is integrable on}\; [\rho(t_0),b]_\T \right\},
\end{equation}
and let
$$ L^2_W\left([\rho(t_0),b]_\T\right):=\left\{y\in L^1\left([\rho(t_0),b]_\T\right): \int_{\rho(t_0)}^{b} \big((\Upsilon y)^*W\Upsilon y \big)(t)\nabla t <\infty \right\} $$
be the space with weighted inner product $\langle\cdot,\cdot\rangle_b$ defined in \eqref{ipb}.

We will use the notation $\|y\|_W=(\langle y,y\rangle)^{1/2}$ for $y\in L^2_W\left([\rho(t_0),\infty)_\T\right)$, and $\|y\|_b=(\langle y,y\rangle_b)^{1/2}$ for $y\in L^2_W\left([\rho(t_0),b]_\T\right)$. As $W$ may be singular, the inner products for $L^2_W\left([\rho(t_0),\infty)_\T\right)$ and $L^2_W\left([\rho(t_0),b]_\T\right)$ may not be positive. To account for this, we introduce the following quotient spaces. For $y,z\in L^2_W\left([\rho(t_0),\infty)_\T\right)$, $y$ and $z$ are said to be equal iff $\|y-z\|_W=0$. In this context $L^2_W\left([\rho(t_0),\infty)_\T\right)$ is an inner product space with inner product $\langle\cdot,\cdot\rangle$. Likewise functions $y,z\in L^2_W\left([\rho(t_0),b]_\T\right)$ are said to be equal iff $\|y-z\|_b=0$, making $L^2_W\left([\rho(t_0),b]_\T\right)$ into an inner product space with inner product $\langle\cdot,\cdot\rangle_b$.


\begin{lemma}
The space $L^2_W\left([\rho(t_0),\infty)_\T\right)$ is a weighted Hilbert space with inner product $\langle\cdot,\cdot\rangle$ given in \eqref{ipinfty}, and $L^2_W\left([\rho(t_0),b]_\T\right)$ is a weighted Hilbert space with inner product $\langle\cdot,\cdot\rangle_b$ given in \eqref{ipb}. In addition, 
$$ \dim L^2_W\left([\rho(t_0),b]_\T\right) = \int_{\rho(t_0)}^b\rank W(t)\nabla t. $$
\end{lemma}

\begin{proof}
The style of proof is based on that given in the discrete case by Shi \cite{shi2}. We will show that $L^2_W\left([\rho(t_0),\infty)_\T\right)$ is complete, as the proof for $L^2_W\left([\rho(t_0),b]_\T\right)$ is similar and is omitted. First, let us consider a simpler case, namely where $W(t)=\nTilde{W}(t)=\diag\{0,W_2(t)\}$ for the $r(t)\times r(t)$ matrix $W_2(t)>0$, $t\in[\rho(t_0),\infty)_\T$ and $0\le r(t)\le 2d$. For convenience, set
$$ \nTilde{W}^{\pm 1/2}(t):=\begin{pmatrix} 0 & 0 \\ 0 & W_2^{\pm 1/2}(t) \end{pmatrix}, $$
$y(t)=\left(y_1^{\trans}(t),y_2^{\trans}(y)\right)^{\trans}$ for $y_i\in\C^d$, $i=1,2$, and $y(t)=\left(y^{(1)\trans}(t),y^{(2)\trans}(t)\right)^{\trans}$ with $y^{(1)}(t)\in\C^{r(t)}$ and $y^{(2)}(t)\in\C^{2d-r(t)}$. 
To prove completeness, assume $\{f_n\}$ is a Cauchy sequence in $L^2_{\nTilde{W}}\left([\rho(t_0),\infty)_\T\right)$, i.e., given $\varepsilon>0$ there exists an $N\in\N$ such that $\|f_n-f_m\|_{\nTilde{W}}<\varepsilon$ for all $n,m\ge N$. Define the functions $g_n:[t_0,\infty)_\T\rightarrow\C^{2d}$ via 	
\begin{equation}\label{shi2.22}
 g_n(t):=\nTilde{W}^{1/2}(t)\Upsilon f_n(t), \qquad n\ge 1, \qquad t\in[t_0,\infty)_\T.
\end{equation}
Since $f_n\in L^2_{\nTilde{W}}\left([\rho(t_0),\infty)_\T\right)$, by \eqref{shi2.22} we have $g_n\in L^2[t_0,\infty)=\left\{g\in L^1[t_0,\infty): (g^*g)\in L^1[t_0,\infty)\right\}$, where $L^2[t_0,\infty)$ is a known Hilbert space with $L^2$ norm $\|g\|_{L^2}=\left(\int_{t_0}^{\infty} (g^*g)(t)\nabla t\right)^{1/2}$ for $g\in L^2[t_0,\infty)$; see \cite[Section 4]{rynne}. It follows that $\|g_n-g_m\|_{L^2} = \|f_n-f_m\|_{\nTilde{W}}$. Thus $\{g_n\}$ is a Cauchy sequence in $L^2[t_0,\infty)$, so by the completeness of $L^2$ there exists an integrable function $g\in L^2[t_0,\infty)$ such that $\|g_n-g\|_{L^2}\rightarrow 0$ as $n\rightarrow\infty$. Set
\begin{equation}\label{shi2.23}
 f(t) = \Upsilon^{-1}(\nTilde{W}^{-1/2}g)(t),
\end{equation}
where $\Upsilon^{-1}$ is given in \eqref{rightshift}. Then 
$\Upsilon f=\Upsilon (\Upsilon^{-1}\nTilde{W}^{-1/2}g)=\nTilde{W}^{-1/2}g$, so that $f\in L^2_{\nTilde{W}}\left([\rho(t_0),\infty)_\T\right)$. To see that $f_n$ converges to $f$ in $L^2_{\nTilde{W}}\left([\rho(t_0),\infty)_\T\right)$, note that by \eqref{shi2.22} we have $g_n^{(1)}(t)=0$ for all $t\in[t_0,\infty)_\T$, whence $g^{(1)}(t)=0$ for $t\in[t_0,\infty)_\T$. As a result, from \eqref{shi2.22} and \eqref{shi2.23} we see that (suppressing the $t$)
\begin{eqnarray*}
 \left[\Upsilon (f_n-f)\right]^*\nTilde{W}\Upsilon (f_n-f) 
 &=& \left(\begin{smallmatrix} 0 \\ W_2^{-1/2}(g_n^{(2)}-g^{(2)}) \end{smallmatrix}\right)^*  
 \left(\begin{smallmatrix} 0 & 0 \\ 0 & W_2 \end{smallmatrix}\right)
 \left(\begin{smallmatrix} 0 \\ W_2^{-1/2}(g_n^{(2)}-g^{(2)}) \end{smallmatrix}\right) \\
 &=& (g_n^{(2)}-g^{(2)})^*(g_n^{(2)}-g^{(2)}),
\end{eqnarray*}
which implies that $\|f_n-f\|_{\nTilde{W}}=\|g_n-g\|_{L^2}$, ergo $f_n\rightarrow f$ in $L^2_{\nTilde{W}}\left([\rho(t_0),\infty)_\T\right)$ as $n\rightarrow\infty$. Consequently, $L^2_{\nTilde{W}}\left([\rho(t_0),\infty)_\T\right)$ is complete.

For the general case, assume $\rank W(t)=r(t)$ for $t\in[\rho(t_0),\infty)_\T$. As $W(t)\ge 0$ and Hermitian, there exists a unitary matrix $U$ such that
$$ U^*(t)W(t)U(t) = \diag\{0,W_2(t)\}=:\nTilde{W}(t), \quad t\in[\rho(t_0),\infty)_\T, $$
where $W_2(t)$ is an $r(t)\times r(t)$ positive definite matrix for all $t\in[\rho(t_0),\infty)_\T$. Suppose $\{f_n\}$ is a Cauchy sequence in $L^2_W\left([\rho(t_0),\infty)_\T\right)$, and set $h_n(t):=\Upsilon^{-1}U^*(\Upsilon f_n)(t)$. Then
\begin{eqnarray*}
 (\Upsilon h_n)^*\nTilde{W}(\Upsilon h_n) &=& (U^*(\Upsilon f_n))^*\nTilde{W}U^*(\Upsilon f_n) \\
 &=& (\Upsilon f_n)^*U\nTilde{W}U^*(\Upsilon f_n) = (\Upsilon f_n)^*W(\Upsilon f_n),
\end{eqnarray*}
so that $h_n\in L^2_{\nTilde{W}}\left([\rho(t_0),\infty)_\T\right)$ and $\|h_n-h_m\|_{\nTilde{W}}=\|f_n-f_m\|_{W}$, making $\{h_n\}$ a Cauchy sequence in $L^2_{\nTilde{W}}\left([\rho(t_0),\infty)_\T\right)$. From the discussion earlier for the simpler case, there exists a function $h\in L^2_{\nTilde{W}}\left([\rho(t_0),\infty)_\T\right)$ such that $h_n\rightarrow h$ in $L^2_{\nTilde{W}}\left([\rho(t_0),\infty)_\T\right)$ as $n\rightarrow\infty$. If we set $f(t)=\Upsilon^{-1}U(\Upsilon h)(t)$, then $f\in L^2_W\left([\rho(t_0),\infty)_\T\right)$ with $\|f_n-f\|_{W}=\|h_n-h\|_{\nTilde{W}}$. It follows that $f_n\rightarrow f$ in $L^2_W\left([\rho(t_0),\infty)_\T\right)$ as $n\rightarrow\infty$, and $L^2_W\left([\rho(t_0),\infty)_\T\right)$ is complete.

To calculate the dimension of $L^2_W\left([\rho(t_0),b]_\T\right)$, note that for $y\in L^2_W\left([\rho(t_0),b]_\T\right)$, $y=0$ if and only if 
$$ \|y\|^2_b=\int_{\rho(t_0)}^b (\Upsilon y)^*(t)W(t)\Upsilon y(t)\nabla t=0, $$
which is equivalent to $W(t)\Upsilon y(t)=0$ for all $t\in[t_0,b]_\T$. Thus $\nTilde{W}U^*(\Upsilon y)(t)=0$ for all $t\in[t_0,b]_\T$, and $U(\Upsilon y)(t)$ has exactly $r(t)$ components taking effect on the inner product. As $U(t)$ is invertible, we have that $\dim L^2_W\left([\rho(t_0),b]_\T\right)=\int_{\rho(t_0)}^br(t)\nabla t$, and the proof is complete. 
\end{proof}


\begin{remark}\label{soboH}
The above result establishes that $L^2_W\left([\rho(t_0),\infty)_\T\right)$ is a Hilbert space with weighted norm 
$$ \|y\|_W^2 = \int_{\rho(t_0)}^{\infty}\big((\Upsilon y)^*W(\Upsilon y)\big)(t)\nabla t. $$
Following \cite[Section 4.1]{rynne}, given $\cldone\left([\rho(t_0),\infty)_\T,\C^{2d}\right)$ with the weighted norm
\begin{equation}\label{Wonenorm}
 \|y\|_1^2:=\|y\|_W^2+\|y^\nabla\|_W^2, \quad y\in\cldone\left([\rho(t_0),\infty)_\T,\C^{2d}\right), 
\end{equation} 
if we define $\mathcal{H}^1_W([\rho(t_0),\infty)_\T)\subset L^2_W\left([\rho(t_0),\infty)_\T\right)$ to be the completion of $\cone\left([\rho(t_0),\infty)_\T,\C^{2d}\right)$ with respect to the norm
$\|\cdot\|_1$ in \eqref{Wonenorm}, then $\mathcal{H}^1_W([\rho(t_0),\infty)_\T)$ is a time-scale analogue of the usual Sobolev space $H^1(I)$ on a real interval
$I$, and $\cldone\left([\rho(t_0),\infty)_\T,\C^{2d}\right)\subset \mathcal{H}^1_W([\rho(t_0),\infty)_\T)$.
\end{remark}

We turn now to definitions of maximal and minimal operators corresponding to system \eqref{maineq}. We use $H$ and $H_0$ to denote the maximal and minimal operators over $[\rho(t_0),\infty)_\T$, respectively, where 
\begin{eqnarray}
 D(H)&:=&\left\{y\in L^2_W\left([\rho(t_0),\infty)_\T\right): \;\text{there exists}\; f\in L^2_W\left([\rho(t_0),\infty)_\T\right)\;\text{such that}\; \right. \nonumber \\
 & & \left. (\mathscr{L}y)(t)=W(t)(\Upsilon f)(t), \; t\in[\rho(t_0),\infty)_\T\right\}, \nonumber \\
 Hy&:=&f, \\
 D(H_0)&:=&\left\{y\in D(H): \;\text{there exists}\; b\in[\rho(t_0),\infty)_\T \;\text{such that}\; \right.\nonumber  \\
 & & \left. y^\rho(t_0)=y(t)=0\;\text{for all}\; t\in[b,\infty)_\T\right\}, \nonumber \\
 H_0y&:=&Hy. \label{andH04.9}
\end{eqnarray}
In a similar manner, we use $H^b$ and $H^b_0$ to denote the maximal and minimal operators over $[\rho(t_0),b]_\T$, respectively, where
\begin{eqnarray}
 D(H^b)&:=&\left\{y\in L^2_W\left([\rho(t_0),b]_\T\right): \;\text{there exists}\; f\in L^2_W\left([\rho(t_0),b]_\T\right)\;\text{such that}\; \right. \nonumber \\
 & & \left. (\mathscr{L}y)(t)=W(t)(\Upsilon f)(t), \; t\in[\rho(t_0),b]_\T\right\}, \nonumber \\
 H^by&:=&f, \label{shi2.25} \\
 D(H^b_0)&:=&\left\{y\in D(H^b): y^\rho(t_0)=y(b)=0\right\}, \nonumber \\
 H^b_0y&:=&H^by. \label{shi2.26}
\end{eqnarray}
By these definitions it is clear that $H_0\subset H$ and $H_0^b\subset H^b$.


\begin{lemma}\label{shilemma2.6}
The operators $H_0$ and $H^b_0$ are Hermitian.
\end{lemma}

\begin{proof}
Since the proof is similar for $H^b_0$, we focus on just $H_0$. For any $y,z\in D(H_0)$, there exists $b\in[\rho(t_0),\infty)_\T$ such that 
\begin{equation}\label{shi2.27} 
 y^\rho(t_0)=z^\rho(t_0)=0, \quad y(t)=z(t)=0 \quad\text{for all} \quad t\in[b,\infty)_\T, 
\end{equation}
and there exists $f,g\in L^2_W\left([\rho(t_0),\infty)_\T\right)$ such that $H_0y=f$ and $H_0z=g$, that is to say
$$ (\mathscr{L}y)(t)=W(t)(\Upsilon f)(t), \quad (\mathscr{L}z)(t)=W(t)(\Upsilon g)(t), \quad t\in[\rho(t_0),\infty)_\T. $$
From Theorem \ref{Lagrange} and \eqref{shi2.27} we have that
\begin{eqnarray*}
 \langle H_0y,z \rangle - \langle y,H_0z \rangle &=& \langle f,z \rangle - \langle y,g \rangle \\
 &=& \int_{\rho(t_0)}^{\infty} \big\{(\Upsilon z)^*W(\Upsilon f) - (\Upsilon g)^*W(\Upsilon y) \big\}(t)\nabla t \\
 &=& \int_{\rho(t_0)}^{\infty} \big\{(\Upsilon z)^*\mathscr{L}y - (\mathscr{L}z)^*\Upsilon y \big\}(t)\nabla t \\
 &=& \lim_{b\rightarrow\infty}z^*(t)Jy(t)\Big|_{\rho(t_0)}^{b}=0.
\end{eqnarray*}
Therefore $\langle H_0y,z \rangle = \langle y,H_0z \rangle$, so that $H_0$ is Hermitian.
\end{proof}

The following lemma has a similar proof to that just completed.


\begin{lemma}\label{shilemma2.7}
The relation $\langle H^b_0y,z \rangle_b = \langle y,H^bz \rangle_b$ holds for all $y\in D(H^b_0)$ and for all $z\in D(H^b)$, while
$\langle H_0y,z \rangle = \langle y,Hz \rangle$ holds for all $y\in D(H_0)$ and for all $z\in D(H)$.
\end{lemma}


\begin{lemma}\label{shilemma2.8}
If \eqref{shiA2} holds, then $\Ran(H^b_0)=\Ker(H^b)^\perp$.
\end{lemma}

\begin{proof}
Given any $f\in\Ran(H_0^b)$, there exists $y\in D(H_0^b)$ such that $H_0^by=f$. For each $z\in\Ker(H^b)$, it follows from the previous lemma that $\langle f,z\rangle_b=\langle H_0^by,z\rangle_b=\langle y,H^b z\rangle_b=\langle y,0\rangle_b=0$, so that $f\in\Ker(H^b)^\perp$. Thus $\Ran(H_0^b)\subset\Ker(H^b)^\perp$. If $f\in\Ker(H^b)^\perp$, then $\langle f,z\rangle_b=0$ for all $z\in\Ker(H^b)$. Consider the following initial value problem:
$$ Jy^\nabla(t)=P(t)\Upsilon y(t)+W(t)\Upsilon f(t), \quad y^\rho(t_0)=0, \qquad t\in[\rho(t_0),b]_\T. $$
By \eqref{shiA2}, this problem has a unique solution $y$ on $[\rho(t_0),b]_\T$. Let $\Phi(t)=(\varphi_1,\varphi_2,\cdots,\varphi_{2d})(t)$ be the fundamental solution matrix of the homogeneous system
$$ Jx^\nabla(t)=P(t)\Upsilon x(t), \quad \Phi(b)=J, \qquad t\in[\rho(t_0),b]_\T. $$
Clearly $\varphi_k\in\Ker(H^b)$ for $1\le k\le 2d$, so by Theorem \ref{Lagrange} and \eqref{ipb},
\begin{eqnarray*}
 0 &=& \langle f,\varphi_k\rangle_b = \int_{\rho(t_0)}^{b} \big\{(\Upsilon \varphi_k)^*W\Upsilon f \big\}(t)\nabla t \\
   &=& \int_{\rho(t_0)}^{b} \big\{(\Upsilon \varphi_k)^*\mathscr{L}y \big\}(t)\nabla t \\
   &=& \int_{\rho(t_0)}^{b} \big\{(\Upsilon \varphi_k)^*\mathscr{L}y-(\mathscr{L}\varphi_k)^*\Upsilon f \big\}(t)\nabla t \\
   &=& \varphi^*_k(b)Jy(b)-\varphi^*_k(\rho(t_0))Jy(\rho(t_0)) \\
   &=& \varphi^*_k(b)Jy(b).
\end{eqnarray*}
Thus we have that $\Phi^*(b)Jy(b)=y(b)=0$, and $\Ker(H^b)^\perp\subset\Ran(H_0^b)$.
\end{proof}


\begin{theorem}
If \eqref{shiA2} holds, then $H_0$ is symmetric and $H$ is densely defined. 
\end{theorem}

\begin{proof}
By Lemma \ref{shilemma2.6} and the fact that $H_0\subset H$, it suffices to show that $D(H_0)$ is dense in $L^2_W\left([\rho(t_0),\infty)_\T\right)$, i.e., that $D(H_0)^\perp=\{0\}$. If $f\in D(H_0)^\perp$, then $D(H^b_0)\subset D(H_0)$ for all $t\in(\rho(t_0),\infty)_\T$ in the sense that all the
functions of $D(H^b_0)$ are considered to have been extended by zero on to $[\rho(t_0),\infty)_\T$. Then, for all $t\in(\rho(t_0),\infty)_\T$ and for all $z\in D(H^b_0)$, 	$\langle f,z\rangle_b=\langle f,z\rangle=0$. Set $H^b_0(z)(t)=g(t)$ for $t\in[\rho(t_0),b]_\T$, and let $y$ be any solution of the system
$$ Jy^\nabla(t) = P(t)\Upsilon y(t)+W(t)\Upsilon f(t), \quad t\in[\rho(t_0),b]_\T. $$
By Theorem \ref{Lagrange}, we have
\begin{eqnarray*}
 \langle y,g\rangle_b-\langle f,z\rangle_b 
 &=& \int_{\rho(t_0)}^{b} \left\{(\Upsilon g)^*W\Upsilon y-(\Upsilon z)^*W\Upsilon f\right\}(t)\nabla t \\
 &=& \int_{\rho(t_0)}^{b} \left\{(\mathscr{L}z)^*\Upsilon y-(\Upsilon z)^*\mathscr{L}y\right\}(t)\nabla t \\
 &=& -z^*(b)Jy(b)+z^*(\rho(t_0))Jy(\rho(t_0))=0.
\end{eqnarray*}
We then have that	$\langle y,g\rangle_b=\langle f,z\rangle_b=0$. It follows that $y\in\Ran(H^b_0)^\perp=\Ker(H^b)$ by Lemma
\ref{shilemma2.8}. Therefore, $H^by=0$, and thus $f\big|_{[\rho(t_0),b]_\T}=0$ in $L^2_W\left([\rho(t_0),b]_\T\right)$, that is $\|f\|_b=0$. Since $b>\rho(t_0)$ is arbitrary, it follows that $f=0$ in $L^2_W\left([\rho(t_0),\infty)_\T\right)$. Consequently, $D(H_0)^\perp=\{0\}$.
\end{proof}


\begin{theorem}
If \eqref{shiA2} holds, then $H_0^*=H$. 
\end{theorem}

\begin{proof}
It suffices to show $D(H_0^*)=D(H)$, and $H_0^*y=Hy$ for all $y\in D(H_0^*)$. If $y\in D(H)$, then $\langle y,H_0z\rangle=\langle Hy,z\rangle$ for all $z\in D(H_0)$ by Lemma \ref{shilemma2.7}. From this we see that the functional $\langle y,H_0(\cdot)\rangle$ is continuous on $D(H_0)$. Then $y\in D(H_0^*)$ by (iv) in Definition \ref{shidef2.1} and thus $D(H)\subset D(H_0^*)$. We now show $D(H_0^*)\subset D(H)$. If $y\in D(H_0^*)$, then $y$ and $g:=H_0^*y$ are both in $L^2_W\left([\rho(t_0),\infty)_\T\right)$. If $x$ is a solution of the system
\begin{equation}\label{shi2.28}
 Jx^\nabla(t) = P(t)\Upsilon x(t) + W(t)\Upsilon g(t), 
\end{equation}
then for each $z\in D(H_0)$ there exists $b\in[\rho(t_0),\infty)_\T$ such that $z^\rho(t_0)=z(t)=0$ for all $t\in[b,\infty)_\T$. Then
$z\big|_{[\rho(t_0),b]_\T}\in D(H^b_0)$. This implies from Lemma \ref{shilemma2.7} that 
$$ \langle g,z\rangle=\langle g,z\rangle_b=\langle H^bx,z\rangle_b=\langle x,H^b_0z\rangle_b=\langle x,H_0z\rangle. $$ 
It follows that 
$$ \langle y-x,H_0z \rangle = \langle y,H_0z \rangle - \langle x,H_0z \rangle = \langle H_0^*y,z\rangle -	\langle g,z\rangle=0. $$
Additionally, $\langle y-x,H^b_0 z\rangle_b=\langle y-x,H_0z\rangle=0$. Thus $(y-x)\big|_{[\rho(t_0),b]_\T}\in\Ran(H^b_0)^\perp$. By Lemma
\ref{shilemma2.8} we have that $(y-x)\big|_{[\rho(t_0),b]_\T}\in\Ker(H^b)$ whereby $H^b(y-x)=0$, in other words, $\mathscr{L}(y-x)(t)=0$ for $t\in[\rho(t_0),b]_\T$. This, together with \eqref{shi2.28}, implies that
$$ Jy^\nabla(t)-P(t)\Upsilon y(t)=Jx^\nabla(t)-P(t)\Upsilon x(t)=W(t)\Upsilon g(t), \quad t\in[\rho(t_0),b]_\T. $$
Since $b\ge\rho(t_0)$ may be chosen arbitrarily large and $y,g\in L^2_W\left([\rho(t_0),\infty)_\T\right)$, we have that $y\in D(H)$ and $Hy=g=H_0^*y$. As a result, $D(H_0^*)\subset D(H)$, whence $D(H_0^*)=D(H)$ and $H^*_0 y=Hy$ for all $y\in D(H_0^*)$.
\end{proof}


\section{Weyl disks and their limiting set}

In this section, we first construct matrix disks for system \eqref{maineq} over time-scale intervals of finite length. These matrix disks are called Weyl disks \cite{cg2}, which turn out to be nested and converge to a limiting set. This limiting set will play a major role in the discussions of square summable solutions of \eqref{maineq}. Again we will rely heavily on the organization of the topic as done by Shi \cite{shi2} in the discrete case.

Suppose that $\theta(t,\lambda)$ and $\phi(t,\lambda)$ are $2d\times d$ matrix-valued solutions of \eqref{maineq} satisfying $\theta(\rho(t_0),\lambda)=\alpha^*$ and $\phi(\rho(t_0),\lambda)= J\alpha^*$, respectively, where $\alpha$ satisfies \eqref{shi2.12}. Then we have from \eqref{shi2.12} that 
\begin{equation}\label{shi3.1}
 \alpha\theta(\rho(t_0),\lambda) = I_d, \quad \alpha\phi(\rho(t_0),\lambda) = 0. 
\end{equation}
Set
\begin{equation}\label{shi3.2}
 Y(t,\lambda):=(\theta,\phi)(t,\lambda),
\end{equation}
so that
\begin{equation}\label{shi3.3}
 Y(\rho(t_0),\lambda)=(\alpha^*,J\alpha^*)=:\Omega.
\end{equation}
Then we have from \eqref{shi2.12} that $\Omega$ is symplectic and unitary; in other words,
\begin{equation}\label{shi3.4}
 \Omega^* J\Omega = J \quad\text{and}\quad \Omega^*\Omega = I_{2d}. 
\end{equation}
Therefore, $Y(\cdot,\lambda)$ is a fundamental solution matrix of \eqref{maineq} and satisfies, by \eqref{shi2.6} and
from \eqref{shi3.4}, that
\begin{equation}\label{shi3.5}
 Y^*\left(t,\overline{\lambda}\right)JY(t,\lambda)=J, \quad t\in[\rho(t_0),\infty)_\T, 
\end{equation}
and thus
\begin{equation}\label{shi3.6}
 Y(t,\lambda)JY^*\left(t,\overline{\lambda}\right)=J,  \quad t\in[\rho(t_0),\infty)_\T.
\end{equation}


\begin{lemma}\label{shilemma3.1}
Let $\alpha$ and $\beta$ satisfy \eqref{shi2.12} and \eqref{shi2.13}, respectively. Then $\lambda$ is an eigenvalue of the problem \eqref{maineq} with boundary conditions \eqref{shi2.11} if and only if $\det(\beta\phi(b,\lambda))=0$. Additionally, $y(t,\lambda)$ is an eigenfunction with respect to $\lambda$ if and only if there exists $\xi\in\C^d$ such that $y(t,\lambda)=\phi(t,\lambda)\xi$, where $\xi\neq 0$ is a solution of the homogeneous linear algebraic system
\begin{equation}\label{shi3.7}
 \beta\phi(b,\lambda)\xi = 0.
\end{equation}
\end{lemma}

\begin{proof}
The proof is unchanged from the discrete case, see \cite[Lemma 3.1]{shi2} and \cite[Lemma 2.8]{cg2}.
\end{proof}

In the subsequent development we will be interested in the function $\chi$ given via
\begin{equation}\label{shi3.8}
 \chi(t,\lambda,b):=Y(t,\lambda)\begin{pmatrix} I_d \\ M(\lambda,b) \end{pmatrix},
\end{equation}
where $M(\lambda,b)$ is a $d\times d$ matrix such that $\beta\chi(b,\lambda,b)=0$, in other words
\begin{equation}\label{shi3.9}
 \beta\theta(b,\lambda)+\beta\phi(b,\lambda)M(\lambda,b)=0.
\end{equation}
Using Lemma \ref{shilemma3.1}, if $\lambda\in\C$ is not an eigenvalue of the problem \eqref{maineq}, \eqref{shi2.11}, then $\beta\phi(b,\lambda)$ is invertible, and from \eqref{shi3.9} we have that
\begin{equation}\label{shi3.10}
 M(\lambda,b) = -\big(\beta\phi(b,\lambda)\big)^{-1}\beta\theta(b,\lambda).
\end{equation}


\begin{lemma}\label{shilemma3.2}
Let $\alpha$ satisfy \eqref{shi2.12}. Then for each $b\in[t_1,\infty)_\T$ $($where $t_1$ is specified in \eqref{shiA1}$)$, we have the following.
\begin{enumerate}
 \item $M(\lambda,b)$ is analytic on the upper and lower half planes and at all non-eigenvalues of the problem \eqref{maineq}, \eqref{shi2.11} on the real axis;
 \item for all $\lambda\in\C$ with $\im\lambda\neq 0$, 
 \begin{equation}\label{shi3.11} M^*\left(\overline{\lambda},b\right) = M(\lambda,b) \end{equation}
 and
 \begin{equation}\label{shi3.12} 
   \im M(\lambda,b):=\frac{M(\lambda,b)-M^*(\lambda,b)}{2i}\lessgtr 0 \quad\text{for}\quad \im \lambda \lessgtr 0. 
 \end{equation}
\end{enumerate}
\end{lemma}

\begin{proof}
The proof is unchanged from the discrete case, see \cite[Lemma 3.2]{shi2} and \cite[Lemma 2.14]{cg2}.
\end{proof}


\begin{lemma}\label{shilemma3.3}
Let $\im\lambda\neq 0$. If $\beta$ satisfies \eqref{shi2.13} and $\chi(t,\lambda,b)$ satisfies
\begin{equation}\label{shi3.13}
 \beta\chi(b,\lambda,b)=0,
\end{equation}
then
\begin{equation}\label{shi3.14}
 \chi^*(b,\lambda,b)J\chi(b,\lambda,b)=0.
\end{equation}
Conversely, if $\chi(t,\lambda,b)$ satisfies \eqref{shi3.14} for some $d\times d$ matrix $M$, then there exists a $d\times 2d$ matrix $\beta$ satisfying \eqref{shi2.13} such that \eqref{shi3.13} holds.
\end{lemma}

\begin{proof}
See \cite[Lemma 2.13]{cg2}.
\end{proof}

Set
$$ C(M,b):=\mp i(I_d,M^*)Y^*(b,\lambda)JY(b,\lambda)\left(\begin{smallmatrix} I_d \\ M \end{smallmatrix}\right), $$
where ``$+$'' holds if $\im \lambda<0$ and ``$-$'' holds when $\im\lambda>0$. Using the definition of $\chi$, it follows that
$$ C\big(M(\lambda,b),b\big)=\mp  i\chi^*(b,\lambda,b)J\chi(b,\lambda,b). $$
Consequently we have from Lemma \ref{shilemma3.3} that $M$ satisfies the matrix equation
\begin{equation}\label{shi3.15}
 C(M,b)=0
\end{equation}
if and only if there exists a $d\times 2d$ matrix $\beta$ satisfying \eqref{shi2.13} such that \eqref{shi3.13} holds, ergo \eqref{shi3.10} holds. Let
\begin{equation}\label{shi3.16}
 F(b,\lambda):=\mp iY^*(b,\lambda)JY(b,\lambda).
\end{equation}
Then $F(b,\lambda)$ is a $2d\times 2d$ Hermitian matrix such that
\begin{equation}\label{shi3.17}
 C(M,b)=(I_d,M^*)F(b,\lambda)\left(\begin{smallmatrix} I_d \\ M \end{smallmatrix}\right).
\end{equation}
For later use we block $F(b,\lambda)$ as in 
\begin{equation}\label{shi3.18}
 F(b,\lambda):=\begin{pmatrix} F_{11} & F_{12} \\ F^*_{12} & F_{22} \end{pmatrix}(b,\lambda),
\end{equation}
where $F_{mn}(b,\lambda)$ are $d\times d$ matrices for $m,n=1,2$. Then \eqref{shi3.15} can be recast in the form
\begin{equation}\label{shi3.19}
 M^*F_{22}(b,\lambda)M+F_{12}(b,\lambda)M+M^*F^*_{12}(b,\lambda)+F_{11}(b,\lambda)=0.
\end{equation}
Using Lemma \ref{lemma2.3}, \eqref{shi3.3}, and \eqref{shi3.4}, we see that
$$ Y^*(b,\lambda)JY(b,\lambda) = J + 2i \im \lambda \int_{\rho(t_0)}^{b}(\Upsilon Y)^*(t,\lambda)W(t)\Upsilon Y(t,\lambda) \nabla t, $$
which in tandem with \eqref{shi3.16} yields
\begin{equation}\label{shi3.20}
 F(b,\lambda)=\begin{cases} -iJ+2\im \lambda \displaystyle \int_{\rho(t_0)}^{b}(\Upsilon Y)^*(t,\lambda)W(t)\Upsilon Y(t,\lambda) \nabla t :& \im\lambda > 0, \\
 iJ-2\im\lambda \displaystyle \int_{\rho(t_0)}^{b}(\Upsilon Y)^*(t,\lambda)W(t)\Upsilon Y(t,\lambda) \nabla t :& \im\lambda < 0. \end{cases}
\end{equation}


\begin{theorem}\label{shithm3.1}
The matrix sets $C(M,b)\le 0$ are closed, convex, and nested in the sense that for fixed $M$ and fixed $\lambda\in\C$, $C^\nabla(M,t)\ge 0$, where the nabla derivative is with respect to $t$.
\end{theorem}

\begin{proof}
The proof follows from \eqref{shi3.17} and \eqref{shi3.20}. See also the discussion in \cite[Remark 2.16]{cg2}.
\end{proof}

As noted in the discrete case \cite[Section 3]{shi2}, which actually refers to \cite[Remark 2.16]{cg2}, the intersection of the matrix sets $C(M,b)\le 0$ is a limiting set that is nonempty, closed, and convex. In what follows here we present a detailed analysis of the properties of $F(b,\lambda)$, which will play a vital role in the next section as we obtain precise relationships among the rank of the matrix radius of the limiting set, asymptotic behavior of eigenvalues of the Weyl disks, and the number of linearly independent square summable solutions of system \eqref{maineq}. Proceeding with this in mind, from \eqref{shi3.16}, \eqref{shi3.18}, and \eqref{shi3.20} we see that
\begin{eqnarray}
 F_{11}(b,\lambda) &=& \mp i\theta^*(b,\lambda)J\theta(b,\lambda) \nonumber \\
 &=& \pm 2\im \lambda\int_{\rho(t_0)}^{b}(\Upsilon \theta)^*(t,\lambda)W(t)\Upsilon \theta(t,\lambda)\nabla t, \nonumber \\
 F_{22}(b,\lambda) &=& \mp i\phi^*(b,\lambda)J\phi(b,\lambda) \label{shi3.21} \\
 &=& \pm 2\im \lambda\int_{\rho(t_0)}^{b}(\Upsilon \phi)^*(t,\lambda)W(t)\Upsilon \phi(t,\lambda)\nabla t, \nonumber \\
 F_{12}(b,\lambda) &=& \mp i\theta^*(b,\lambda)J\phi(b,\lambda) \nonumber \\
 &=& \pm iI_d \pm 2\im \lambda \int_{\rho(t_0)}^{b}(\Upsilon \theta)^*(t,\lambda)W(t)\Upsilon \phi(t,\lambda)\nabla t. \nonumber 
\end{eqnarray}

Assuming \eqref{shiA1}, we can obtain the following from \eqref{shi3.21}.

\begin{theorem}\label{shiprop3.1}
For any $\lambda\in\C$ with $\im\lambda\neq 0$, $F_{11}(b,\lambda)>0$ and $F_{22}(b,\lambda)>0$ for all $b\ge t_1$. Additionally, $F_{11}(b,\lambda)$ and $F_{22}(b,\lambda)$ are non-decreasing with respect to $b$.
\end{theorem}

Using Theorem \ref{shiprop3.1}, equation \eqref{shi3.19} can be stated as
\begin{equation}\label{shi3.22}
 \left(M+F_{22}^{-1}(b,\lambda)F_{12}^*(b,\lambda)\right)^* F_{22}(b,\lambda)\left(M+F_{22}^{-1}(b,\lambda)F_{12}^*(b,\lambda)\right) - \left(F_{12}F_{22}^{-1}F^*_{12}-F_{11}\right)(b,\lambda)=0.
\end{equation}

\begin{theorem}\label{shiprop3.2}
For any $\lambda\in\C$ with $\im\lambda\neq 0$, $(F_{12}F_{22}^{-1}F^*_{12}-F_{11})(b,\lambda)=F_{22}^{-1}\left(b,\overline{\lambda}\right)>0$.
\end{theorem}

\begin{proof}
From \eqref{shi3.5}, \eqref{shi3.6}, and \eqref{shi3.16} we see that
$$ F(b,\lambda)JF\left(b,\overline{\lambda}\right)  = Y^*(b,\lambda)JY(b,\lambda)JY^*\left(b,\overline{\lambda}\right)JY\left(b,\overline{\lambda}\right) = -J. $$
The rest of the proof is identical to \cite[Proposition 3.2]{shi2} and is omitted.
\end{proof}

Let
\begin{equation}\label{shi3.23}
 \mathscr{C}(b,\lambda):= -F_{22}^{-1}(b,\lambda)F_{12}^*(b,\lambda), \quad \mathscr{R}(b,\lambda):= F_{22}^{-1/2}(b,\lambda).
\end{equation}
It follows from Theorem \ref{shiprop3.2} that \eqref{shi3.15}, that is to say \eqref{shi3.22}, can be recast as
\begin{equation}\label{shi3.24}
 (M-\mathscr{C}(b,\lambda))^*\mathscr{R}^{-1}(b,\lambda)(M-\mathscr{C}(b,\lambda))-\mathscr{R}^2\left(b,\overline{\lambda}\right)=0
\end{equation}
or as 
\begin{equation}\label{shi3.25}
 \left\{\mathscr{R}^{-1}(b,\lambda)(M-\mathscr{C}(b,\lambda))\mathscr{R}^{-1}\left(b,\overline{\lambda}\right)\right\}^*\left\{\mathscr{R}^{-1}(b,\lambda)(M-\mathscr{C}(b,\lambda))\mathscr{R}^{-1}\left(b,\overline{\lambda}\right)\right\}=I_{d}.
\end{equation}

\begin{remark}\label{shiremark3.2}
If the dimension $d=1$, then \eqref{shi3.24} is the equation of a circle. For this reason we call \eqref{shi3.15} and/or \eqref{shi3.24} a Weyl circle equation, and call the matrix set $C(M,b)\le 0$ a Weyl disk; see \cite[Definition 2.11]{cg2}.
\end{remark}

Notice that if $U:=\mathscr{R}^{-1}(b,\lambda)(M-\mathscr{C}(b,\lambda))\mathscr{R}^{-1}\left(b,\overline{\lambda}\right)$, then \eqref{shi3.25} can be written as $U^*U=I_d$, and $U$ is unitary. We then have the following results.

\begin{theorem}\label{shiprop3.3}
The Weyl circle equation \eqref{shi3.24} and/or \eqref{shi3.15} can be expressed via
\begin{equation}\label{shi3.26}
 E_b(\lambda): M=\mathscr{C}(b,\lambda)+\mathscr{R}(b,\lambda)U\mathscr{R}\left(b,\overline{\lambda}\right),
\end{equation}
and the Weyl disk $C(M,b)\le 0$ can be expressed via
\begin{equation}\label{shi3.27}
 \overline{E}_b(\lambda): M=\mathscr{C}(b,\lambda)+\mathscr{R}(b,\lambda)V\mathscr{R}\left(b,\overline{\lambda}\right),
\end{equation}
where $U$ is any matrix on the unit matrix circle $\partial D = \{U:U\in\C^{d\times d}\;\text{is a unitary matrix}\}$ and $V$ is any matrix on the unit matrix disk $D=\{V:V\in\C^{d\times d}\; \text{satisfies}\; V^*V\le I_d\}$.
\end{theorem}


\begin{definition}
The matrix $\mathscr{C}(b,\lambda)$ is called the center, and the matrices $\mathscr{R}(b,\lambda)$ and $\mathscr{R}\left(b,\overline{\lambda}\right)$ are called the matrix radii, respectively, of the Weyl circle \eqref{shi3.26} and the Weyl disk \eqref{shi3.27}.
\end{definition}


\begin{theorem}\label{shiprop3.4}
For any given $\lambda\in\C$ with $\im\lambda\neq 0$, the matrix sequence $\{\mathscr{R}(b,\lambda)\}$ converges, and $\mathscr{R}_0(\lambda):=\lim_{b\rightarrow \infty}\mathscr{R}(b,\lambda)\ge 0$.
\end{theorem}

\begin{proof}
From Theorem \ref{shiprop3.1}, $F_{22}(b,\lambda)>0$ for all $b\ge t_1$, and $F_{22}(b,\lambda)$ is non-decreasing with respect to $b$. Recall from \eqref{shi3.23} that $\mathscr{R}(b,\lambda)=F_{22}^{-1/2}(b,\lambda)$, whence $\{\mathscr{R}(b,\lambda)\}$ is a non-increasing sequence of positive definite matrices. The conclusion follows from the fact that any non-increasing sequence of Hermitian matrices that is bounded below converges to a Hermitian matrix.
\end{proof}


\begin{lemma}\label{shilemma3.4}
For any given $\lambda\in\C$ with $\im\lambda\neq 0$, and $b>\tau\ge t_1$, there exists $V_0\in D$ such that
\begin{equation}\label{shi3.28}
 \mathscr{C}(b,\lambda)-\mathscr{C}(\tau,\lambda) = \mathscr{R}(\tau,\lambda)V_0\mathscr{R}\left(\tau,\overline{\lambda}\right) - \mathscr{R}(b,\lambda)V_0\mathscr{R}\left(b,\overline{\lambda}\right).
\end{equation}
\end{lemma}

\begin{proof}
See Shi \cite[Lemma 3.4]{shi2}.
\end{proof}


\begin{theorem}\label{shiprop3.5}
For any given $\lambda\in\C$ with $\im\lambda\neq 0$, the matrix sequence $\{\mathscr{C}(b,\lambda)\}$ converges, i.e. the matrix $\mathscr{C}_0(\lambda):=\lim_{b\rightarrow\infty}\mathscr{C}(b,\lambda)$ is well defined.
\end{theorem}

\begin{proof}
The result follows from Theorem \ref{shiprop3.4} and Lemma \ref{shilemma3.4}.
\end{proof}


\begin{theorem}\label{shithm3.2}
For any given $\lambda\in\C$ with $\im\lambda\neq 0$, the matrix circle sequence $\{E_b(\lambda)\}$ and the matrix disk sequence $\left\{\overline{E}_b(\lambda)\right\}$ converge as $b\rightarrow\infty$, and their limiting sets can be represented, respectively, as
\begin{eqnarray*}
 E_0(\lambda): & M = \mathscr{C}_0(\lambda) + \mathscr{R}_0(\lambda)U\mathscr{R}_0\left(\overline{\lambda}\right), \quad U\in\partial D, \\
 \overline{E}_0(\lambda): & M = \mathscr{C}_0(\lambda) + \mathscr{R}_0(\lambda)V\mathscr{R}_0\left(\overline{\lambda}\right), \quad V\in D.
\end{eqnarray*}
\end{theorem}

\begin{proof}
 The result follows from Theorems \ref{shiprop3.3}, \ref{shiprop3.4}, and \ref{shiprop3.5}.
\end{proof}


\begin{remark}
Since $\mathscr{R}_0(\lambda)$ and $\mathscr{R}_0\left(\overline{\lambda}\right)$ may be singular, the set $\overline{E}_0(\lambda)$ may be a reduced matrix disk. We see that $\overline{E}_0(\lambda)$ contains only one element if $\mathscr{R}_0(\lambda)=0$ or $\mathscr{R}_0\left(\overline{\lambda}\right)=0$, and it contains interior points if and only if $\mathscr{R}_0(\lambda)$ and $\mathscr{R}_0\left(\overline{\lambda}\right)$ are both invertible. Although the limiting sets $E_0(\lambda)$ and $\overline{E}_0(\lambda)$ may be a reduced matrix circle and a reduced
matrix disk, respectively, we still give the following definition for convenience.
\end{remark}


\begin{definition}\label{shidef3.2}
The matrix $\mathscr{C}_0(\lambda)$ is called the center, and the matrices $\mathscr{R}_0(\lambda)$ and $\mathscr{R}_0\left(\overline{\lambda}\right)$ are
called the matrix radii of the limiting sets $E_0(\lambda)$ and $\overline{E}_0(\lambda)$, respectively.
\end{definition}


\begin{theorem}\label{shithm3.3}
Assume \eqref{shiA1}. For any given $\lambda\in\C$ with $\im\lambda\neq 0$ and for each $M\in\overline{E}_0(\lambda)$, if $\im\lambda\lessgtr 0$ then $\im M\lessgtr 0$.
\end{theorem}

\begin{proof}
Assume that  $\im\lambda\neq 0$ and $M\in\overline{E}_0(\lambda)$. Set
\begin{equation}\label{shi3.30}
 \chi(t,\lambda):=Y(t,\lambda)\left(\begin{smallmatrix}I_d \\ M \end{smallmatrix}\right).
\end{equation}
Then from \eqref{shi3.20} we have that
\begin{eqnarray}
 \int_{\rho(t_0)}^{b} (\Upsilon \chi)^*(t,\lambda)W(t)\Upsilon \chi(t,\lambda)\nabla t 
 &=& (I_d,M^*) \int_{\rho(t_0)}^{b} (\Upsilon Y)^*(t,\lambda)W(t)\Upsilon Y(t,\lambda)\nabla t \left(\begin{smallmatrix}I_d \\ M \end{smallmatrix}\right) \nonumber \\
 &=& \frac{1}{2| \im\lambda|}(I_d,M^*)\left(F(b,\lambda)\pm i J\right) \left(\begin{smallmatrix}I_d \\ M \end{smallmatrix}\right) \nonumber  \\
 &=& \frac{1}{2| \im\lambda|}(I_d,M^*)F(b,\lambda)\left(\begin{smallmatrix}I_d \\ M \end{smallmatrix}\right) \pm \frac{1}{|\im\lambda|}\im M, \label{shi3.31}
\end{eqnarray}
for $\im\lambda\lessgtr 0$. Because the sets $C(M,b)\le 0$ are nested, $\overline{E}_0(\lambda)$ is a subset of $C(M,b)\le 0$ for any $b\ge t_1$. Consequently we have from \eqref{shi3.17} that
$$ (I_d,M^*) F(b,\lambda) \left(\begin{smallmatrix}I_d \\ M \end{smallmatrix}\right)\le 0. $$
This in tandem with \eqref{shi3.31} implies that for $b\ge t_1$ we have
\begin{equation}\label{shi3.32}
 \int_{\rho(t_0)}^{b} (\Upsilon \chi)^*(t,\lambda)W(t)\Upsilon \chi(t,\lambda)\nabla t \le \pm \frac{1}{|\im\lambda|}\im M, \quad \im\lambda\lessgtr 0.
\end{equation}
The result then follows from the above relation and the assumed definiteness condition \eqref{shiA1}.
\end{proof}


\section{Square summable solutions}

We will call $y(\cdot,\lambda)$ a square summable solution of \eqref{maineq} if it is a solution of \eqref{maineq} in $L^2_W\left([\rho(t_0),\infty)_\T\right)$. In this section we will make a connection between square summable solutions of \eqref{maineq} and the elements of the limiting set $\overline{E}_0(\lambda)$ from Theorem \ref{shithm3.2}, obtaining precise relationships among the rank of the matrix radius $\mathscr{R}_0(\lambda)=\lim_{b\rightarrow \infty}F_{22}^{-1/2}(b,\lambda)$ (see Theorem \ref{shiprop3.4}) of the limiting set  $\overline{E}_0(\lambda)$, the number of linearly independent square summable solutions of \eqref{maineq}, and the asymptotic behavior of the eigenvalues of the matrix radius $F_{22}(b,\lambda)$ of the Weyl disk $\overline{E}_b(\lambda)$ from \eqref{shi3.27}. Given the structure and notation established in the previous sections that generalizes the discrete results in Shi \cite{shi2}, the proofs of the following results are omitted, as there is no change necessary from \cite[Section 4]{shi2} except for minor notational adjustments. 


\begin{theorem}
For each $\lambda\in\C$ with $\im\lambda\neq 0$ and for each $M\in\overline{E}_0(\lambda)$, all the columns of $\chi(\cdot,\lambda)$ are in $L^2_W\left([\rho(t_0),\infty)_\T\right)$, where $\chi$ is given in \eqref{shi3.30}.
\end{theorem}

\begin{cor}\label{shicor4.1}
For each $\lambda\in\C$ with $\im\lambda\neq 0$, system \eqref{maineq} has at least $d$ linearly independent square summable solutions.
\end{cor}


\begin{theorem}
For $\mathscr{R}_0(\lambda)=\lim_{b\rightarrow \infty}F_{22}^{-1/2}(b,\lambda)$, set
$$ r(\lambda):=\rank\mathscr{R}_0(\lambda), \quad \im\lambda\neq 0, $$
and let $k=d+\min\left\{r(\lambda),r\left(\overline{\lambda}\right)\right\}$. Then system \eqref{maineq} has at least $k$ linearly independent square summable solutions.
\end{theorem}

By Theorem \ref{shiprop3.1}, for each $\lambda\in\C$ with $\im\lambda\neq 0$ and for any $b\ge t_1$, $F_{22}(b,\lambda)>0$. If $\mu_j(b)$ are the eigenvalues of $F_{22}(b,\lambda)$ for $1\le j\le d$, then $\mu_j(b)>0$ for $1\le j\le d$, and they can be arranged as $\mu_1(b)\le\mu_2(b)\le\cdots\le\mu_d(b)$. We have the following result.


\begin{theorem}\label{shithm4.3}
For each $\lambda\in\C$ with $\im\lambda\neq 0$, $\rank \mathscr{R}_0(\lambda)=r(\lambda)$ if and only if 
$$ \lim_{b\rightarrow\infty}\mu_j(b)=\gamma_j, \quad 1\le j\le r(\lambda), $$ 
are finite and positive, and 
$$ \lim_{b\rightarrow\infty}\mu_j(b)=\infty, \quad r(\lambda)+1\le j\le d. $$ 
In addition, $\gamma_1^{-1/2},\gamma_2^{-1/2},\cdots,\gamma_{r(\lambda)}^{-1/2}$ are the $r(\lambda)$ positive eigenvalues of $\mathscr{R}_0(\lambda)$.
\end{theorem}


\begin{lemma}
For each $\lambda\in\C$ with $\im\lambda\neq 0$, system \eqref{maineq} has exactly $d+l$ linearly independent square summable solutions if and only if there exists a $d\times l$ matrix $\Lambda$ with $\rank\Lambda=l$ such that $\phi(\cdot,\lambda)\Lambda\in L^2_W\left([\rho(t_0),\infty)_\T\right)$, and $\eta\in\Ran\Lambda=\left\{\Lambda v:v\in\C^l\right\}$ if $\phi(\cdot,\lambda)\eta\in L^2_W\left([\rho(t_0),\infty)_\T\right)$ for some $\eta\in\C^d$.
\end{lemma}


\begin{theorem}\label{shithm4.4}
If $\rank\mathscr{R}_0(\lambda)=r(\lambda)$ for $\im\lambda\neq 0$, then system \eqref{maineq} has exactly $d+r(\lambda)$ linearly independent square summable solutions and thus, $r(\lambda)$ is independent of the coefficient matrix $\alpha$ of the left boundary condition in \eqref{shi2.11}.
\end{theorem}


\begin{theorem}
For each $\lambda\in\C$ with $\im\lambda\neq 0$, system \eqref{maineq} has exactly $d+r(\lambda)$ linearly independent square summable solutions if and only if $\lim_{b\rightarrow\infty}\mu_j(b)=\gamma_j$ are finite and positive for $1\le j\le r(\lambda)$, and $\lim_{b\rightarrow\infty}\mu_j(b)=\infty$ for $r(\lambda)+1\le j \le d$. In addition, $\gamma_1^{-1/2},\gamma_2^{-1/2},\cdots,\gamma_{r(\lambda)}^{-1/2}$ are the $r(\lambda)$ positive eigenvalues of $\mathscr{R}_0(\lambda)$.
\end{theorem}


\section{Classification of singular linear Hamiltonian nabla systems}

In this section, we introduce the defect index $d(\lambda)$ of the minimal operator $H_0$ (defined in Section $4$) and $\lambda$. We establish a precise correspondence between $d(\lambda)$ and the number of linearly independent square summable solutions of \eqref{maineq}. Based on this correspondence, we show that the defect indices $d_{\pm}$ of $H_0$ are not less than $d$. In addition, we obtain a precise correspondence between
$d(\lambda)$ and $\rank\mathscr{R}_0(\lambda)$. Moreover, we discuss the defect index problem for the special case where $P(t)$ and $W(t)$ are both real, and the largest defect index problem for the general case. Building on the above results, we present a suitable classification for singular Hamiltonian nabla systems by using the positive and negative defect indices of $H_0$. Lastly, we derive several equivalent conditions on the limit circle and the limit point cases. The proofs of the first few results below carry over from the discrete case unchanged \cite[Section 5]{shi2}.


\begin{theorem}\label{shithm5.1}
For all $\lambda\in\C$, the defect index $d(\lambda)$ of the minimal operator $H_0$ and $\lambda$ is equal to
the number of linearly independent square summable solutions of system \eqref{maineq}.
\end{theorem}


\begin{remark}
As $H_0$ is Hermitian by Lemma $\ref{shilemma2.6}$, the defect index $d(\lambda)$ of $H_0$ is constant in the upper and lower half
planes, respectively. Set
\begin{equation}\label{shi5.2}
 d_+ = d(i), \quad d_- = d(-i); 
\end{equation}
these are called the positive and negative defect indices of the minimal operator, respectively. The following result then follows directly from Corollary $\ref{shicor4.1}$ and Theorem $\ref{shithm5.1}$.
\end{remark}


\begin{theorem}\label{shithm5.2}
The number of linearly independent square summable solutions of system \eqref{maineq} in the upper half plane is $d_+$, and in the lower half plane is $d_-$. These numbers are independent of $\lambda$, with $d_{\pm}\ge d$.  
\end{theorem}


\begin{theorem}\label{shithm5.3}
The rank of $\mathscr{R}_0(\lambda)$ is equal to a constant $r_+$ for all $\lambda$ with $\im\lambda > 0$, and equal to a constant $r_-$ for all $\lambda$ with $\im\lambda < 0$. Moreover, these ranks satisfy the equations
\begin{equation}\label{shi5.3}
 d_+ = d + r_+, \qquad d_- = d + r_-.
\end{equation}
\end{theorem}


\begin{lemma}
For any $\lambda_{0}\in\C$, the fundamental matrix solution $\Phi(\cdot,\lambda_0)$ of \eqref{maineq}$_{\lambda_0}$ with initial condition $\Phi(\rho(t_0),\lambda_0)=I_{2d}$ satisfies
\begin{equation}\label{Phinormal}
  \det\left(\Phi^{*}(t,\lambda_0)\Phi(t,\lambda_0)\right) = 1 
\end{equation}
for all $t\in[\rho(t_0),\infty)_\T$.
\end{lemma}

\begin{proof}
By the initial condition, \eqref{Phinormal} holds at $t=\rho(t_0)$. If $t\in[t_0,\infty)_\T$ is a left-scattered point, then from \eqref{nabsym} we have
\begin{equation}\label{shi5.9temp}
 \left(I_{2d}-\nu(t)\mathcal{S}(t,\lambda_0)\right)\Phi(t,\lambda_0) = \Phi^\rho(t,\lambda_0). 
\end{equation}
Consequently by \eqref{imsstar} and \eqref{shi5.9temp} we have $\det\big(\Phi^{*}(t,\lambda_0)\Phi(t,\lambda_0)\big) = \det\big(\Phi^{\rho*}(t,\lambda_0)\Phi^{\rho}(t,\lambda_0)\big)$, so that 
$$ \left[\det\left(\Phi^{*}(t,\lambda_0)\Phi(t,\lambda_0)\right)\right]^{\nabla}(t)=0 $$
if $t$ is a left-scattered point.
 
By Liouville's formula on time scales \cite{cormani}, we have
\begin{equation}\label{LiouvilleThm}
 \det\Phi(t,\lambda_0) = \hat{e}_{q}(t,\rho(t_0)) \det \Phi(\rho(t_0),\lambda_0)=\hat{e}_{q}(t,\rho(t_0)), \quad t\in[\rho(t_0),\infty)_\T,
\end{equation}
where $q(t)=\lambda_1\oplus_{\nu}\lambda_2\oplus_{\nu}+\cdots+\oplus_{\nu}\lambda_{2d}$ for eigenvalues $\lambda_1,\cdots,\lambda_{2d}$ of $\mathcal{S}(\cdot,\lambda)$ given in \eqref{nabsym}, and $x=\hat{e}_q(\cdot,\rho(t_0))$ is the nabla exponential function \cite[Chapter 3]{bp2} that uniquely solves the initial value problem
$$ x^\nabla(t)=q(t)x(t), \quad x^\rho(t_0)=1. $$
It follows that $\det\left(\Phi^{*}(t,\lambda_0)\Phi(t,\lambda_0)\right)=\hat{e}_{(q\oplus_{\nu}q^*)}(t,\rho(t_0))$.
Suppose $t\in[t_0,\infty)_\T$ is a left-dense point. Then $\nu(t)=0$, $(q\oplus_{\nu}q^*)=(q+q^*)=\tr (\mathcal{S}+\mathcal{S}^*)=0$ from \eqref{nabsym}, and
$$ \det\left(\Phi^{*}(t,\lambda_0)\Phi(t,\lambda_0)\right) = \hat{e}_{(q+q^*)}(t,\rho(t_0)) = \hat{e}_{0}(t,\rho(t_0))\equiv 1.  $$
Therefore \eqref{Phinormal} holds for all $t\in[\rho(t_0),\infty)_\T$.
\end{proof}


\begin{theorem}[The Largest Defect Index Theorem]\label{shithm5.5}
If there exists $\lambda_0\in\C$ such that all the solutions of \eqref{maineq}$_{\lambda_{0}}$ are in $L^2_W\left([\rho(t_0),\infty)_\T\right)$, then all solutions of \eqref{maineq}$_{\lambda}$ are in $L^2_W\left([\rho(t_0),\infty)_\T\right)$, for any $\lambda\in\C$.
\end{theorem}

\begin{proof}
Assume that all solutions of \eqref{maineq}$_{\lambda_{0}}$ are in $L^2_W\left([\rho(t_0),\infty)_\T\right)$ for some $\lambda_0\in\C$. Given any $\lambda\in\C$, let $\Phi(\cdot,\lambda)$ be the fundamental matrix solution of \eqref{maineq}$_{\lambda}$, that is $\Phi(\cdot,\lambda)$ solves \eqref{maineq}$_\lambda$, or \eqref{shi2.1}$_\lambda$, respectively, and satisfies $\Phi(\rho(t_0),\lambda)=I_{2d}$; note that \eqref{Phinormal} then holds. As $\Phi(t,\lambda)$ and $\Phi\left(t,\lambda_0\right)$ are both invertible, there exists an invertible matrix $X(t,\lambda)$ such that
\begin{equation}\label{shi5.4}
 \Phi(t,\lambda) = \Phi\left(t,\lambda_0\right)X(t,\lambda).
\end{equation}
We will show that $X(t,\lambda)$ is bounded for all $t\in[\rho(t_0),\infty)_\T$. Using \eqref{shi5.4}, the nabla product rule, and the simple useful time-scale formula $X^\rho=X-\nu  X^\nabla$,  we see that
\begin{eqnarray} 
 &&\Phi^\nabla(t,\lambda) = \Phi^\rho\left(t,\lambda_0\right) X^\nabla(t,\lambda) + \Phi^\nabla\left(t,\lambda_0\right) X(t,\lambda), \nonumber \\
 &&  (\Upsilon\Phi)(t,\lambda) = (\Upsilon\Phi)\left(t,\lambda_0\right) X(t,\lambda) - \nu(t) \diag\{0,I_d\} \Phi^\rho\left(t,\lambda_0\right) X^\nabla(t,\lambda), \label{shi5.5}
\end{eqnarray}
for $\Upsilon$ in \eqref{leftshift}. From \eqref{shi5.5} and the fact that $\Phi(\cdot,\lambda)$ and $\Phi(\cdot,\lambda_0)$ are fundamental solution matrices for \eqref{maineq}$_\lambda$ and \eqref{maineq}$_{\lambda_0}$, respectively, we arrive at
\begin{equation}\label{shi5.6}
 X^\nabla(t,\lambda) = Q(t,\lambda)X(t,\lambda),
\end{equation}
where we have taken 
\begin{eqnarray}
 Q(t,\lambda) 
 &=& (\lambda-\lambda_0) Z^{-1}(t,\lambda)(\Upsilon\Phi)^*(t,\lambda_0)W(t)(\Upsilon\Phi)(t,\lambda_0), \label{shi5.7} \\
 Z(t,\lambda) 
 &=& (\Upsilon\Phi)^{*}(t,\lambda_0) \begin{pmatrix} 0 & -I_{d}+\nu(t)A^*(t) \\ 
       I_d & \nu(t)(B(t)+\lambda W_2(t)) \end{pmatrix} \Phi^{\rho}(t,\lambda_0). \label{shi5.8}
\end{eqnarray} 
The multiplier $(\Upsilon\Phi)^*(t,\lambda_0)$ appears in \eqref{shi5.7} and \eqref{shi5.8} via \eqref{rightshift}, and will help in the sequel with the analysis on $Q(t,\lambda)$. First we focus on $Z(t,\lambda)$.  From \eqref{maineq}$_{\lambda_0}$ we have that
\begin{equation}\label{shi5.9}
 \Phi^\rho(t,\lambda_0) = \begin{pmatrix} I_d-\nu(t)A(t) & -\nu(t)(B(t)+\lambda_0W_2(t)) \\ 0 & I_d \end{pmatrix}(\Upsilon\Phi)(t,\lambda_0).
\end{equation}
If we substitute \eqref{shi5.9} into \eqref{shi5.8}, we see that
\begin{eqnarray}
 Z(t,\lambda) &=& (\lambda-\lambda_0)\int_{\rho(t)}^{t}(\Upsilon\Phi)^{*}(s,\lambda_0)\diag\{0,W_2(s)\} (\Upsilon\Phi)(s,\lambda_0)\nabla s   \nonumber \\
 & & + (\Upsilon\Phi)^{*}(t,\lambda_0) \begin{pmatrix} 0 & -I_{d}+\nu(t)A^*(t) \\ 
       I-\nu(t)A(t) & 0 \end{pmatrix} (\Upsilon\Phi)(t,\lambda_0), \label{shi5.10}
\end{eqnarray}
where we have used the time-scale formula $\nu(t)f(t)=\int_{\rho(t)}^{t}f(s)\nabla s$ in the first line of \eqref{shi5.10}. As all solutions of \eqref{maineq}$_{\lambda_0}$ are in $L^2_W\left([\rho(t_0),\infty)_\T\right)$, $\Phi(\cdot,\lambda_0)\in L^2_W\left([\rho(t_0),\infty)_\T\right)$, whence
\begin{equation}\label{shi5.11}
 V(\lambda_0):=\int_{\rho(t_0)}^{\infty} (\Upsilon\Phi)^{*}(s,\lambda_0) W(s) (\Upsilon\Phi)(s,\lambda_0) \nabla s < \infty.
\end{equation}
Consequently the first term on the right-hand side of \eqref{shi5.10} tends to zero as $t\rightarrow\infty$ in the time scale. Let us denote by $\Psi$ the second term on the right-hand side of \eqref{shi5.10}. From \eqref{shi5.9} we have
\begin{equation}\label{Psidef}
 \Psi(t) = \Phi^{\rho*}(t,\lambda_0)J\Phi^\rho(t,\lambda_0) + 2i\im\lambda_0 \int_{\rho(t)}^{t}(\Upsilon\Phi)^{*}(s,\lambda_0)\diag\{0,W_2(s)\} (\Upsilon\Phi)(s,\lambda_0)\nabla s,
\end{equation}
where we have used the time-scale formula $\nu(t)f(t)=\int_{\rho(t)}^{t}f(s)\nabla s$ again; the second term on the right-hand side of \eqref{Psidef} goes to 0 as $t\rightarrow\infty$ in the time scale since $\Phi(\cdot,\lambda_0)\in L^2_W\left([\rho(t_0),\infty)_\T\right)$. By Lemma \ref{lemma2.3} and the initial condition for $\Phi(\cdot,\lambda_0)$ we see that
$$ \Phi^{\rho*}(t,\lambda_0)J\Phi^\rho(t,\lambda_0) = J +2i\im\lambda_0 \int_{\rho(t_0)}^{\rho(t)} (\Upsilon\Phi)^{*}(s,\lambda_0) W(s) (\Upsilon\Phi)(s,\lambda_0) \nabla s. $$
It then follows from \eqref{shi5.11} and \eqref{Psidef} that 
$$ \lim_{t\rightarrow\infty} \Psi(t) = J+2i\im\lambda_0 V(\lambda_0), $$ 
so that
\begin{equation}\label{shi5.12}
 \lim_{t\rightarrow\infty} Z(t,\lambda) = J+2i\im\lambda_0 V(\lambda_0).
\end{equation}
From \eqref{shi5.8} and \eqref{shi5.9} we see that $Z(t,\lambda)$ is invertible for all $t\in[\rho(t_0),\infty)_\T$; we need to show that $Z^{-1}(t,\lambda)$ is bounded for all $t\in[\rho(t_0),\infty)_\T$. From \eqref{shi5.9} we have that 
$$ \det(\Upsilon\Phi)^*(t,\lambda_0) =  \det \Phi^{\rho*}(t,\lambda_0) \det E^*(t), $$
so that from \eqref{shi5.8} we obtain
\begin{eqnarray*}
 \det Z(t,\lambda) &=& \det(\Upsilon\Phi)^{*}(t,\lambda_0) \det\left(I_d-\nu(t)A^*(t)\right) \det \Phi^{\rho}(t,\lambda_0) \\
 &=&  \det\Phi^{\rho*}(t,\lambda_0)  \det\Phi^{\rho}(t,\lambda_0) = 1
\end{eqnarray*} 
since $\Phi(\cdot,\lambda_0)$ satisfies \eqref{Phinormal}. As a result, 
\begin{equation}\label{shi5.13}
  Z^{-1}(t,\lambda) = (\det Z(t,\lambda))^{-1}\adj Z(t,\lambda),
\end{equation}
where $\adj Z(t,\lambda)$ is the adjugate matrix of $Z(t,\lambda)$. Moreover, from \eqref{shi5.12} we see that $\adj Z(t,\lambda)$ is bounded on $[\rho(t_0),\infty)_\T$, whence $Z^{-1}(t,\lambda)$ is as well by \eqref{shi5.13}. Let $c\in\R$ be a positive constant such that
\begin{equation}\label{shi5.14}
 \|Z^{-1}(t,\lambda)\|:=\left(\sum_{k=1}^{2d}\sum_{j=1}^{2d}|z_{jk}(t)|\right)^{1/2}\le c, \quad t\in[\rho(t_0),\infty)_\T.
\end{equation}
Now we will show that
\begin{equation}\label{shi5.15}
 \int_{\rho(t_0)}^{\infty} \|Q(t,\lambda) \|_1\nabla t < \infty, \quad \|Q(t,\lambda) \|_1:=\sup_{\|\xi\|=1}\|Q(t,\lambda)\xi \|,
\end{equation}
where $Q(\cdot,\lambda)$ is given in \eqref{shi5.7}. From \eqref{shi5.11} it follows that all the diagonal entries of the expression \begin{equation}\label{recurringterm}
 (\Upsilon\Phi)^{*}(t,\lambda_0) W(t) (\Upsilon\Phi)(t,\lambda_0)
\end{equation}
are nonnegative and absolutely summable over $[\rho(t_0),\infty)_\T$. By referring to the nonnegativity of \eqref{recurringterm}, the absolute value of each non-diagonal entry of \eqref{recurringterm} is less than or equal to the sum of the two diagonal entries that lie exactly in the same column and row as the non-diagonal entry does. As a result, each non-diagonal entry of \eqref{recurringterm} is also absolutely summable over $[\rho(t_0),\infty)_\T$. Thus, it follows that
\begin{equation}\label{shi5.16}
 \int_{\rho(t_0)}^{\infty} \|(\Upsilon\Phi)^{*}(t,\lambda_0) W(t) (\Upsilon\Phi)(t,\lambda_0)\|\nabla t < \infty.
\end{equation}
Consequently from \eqref{shi5.7}, \eqref{shi5.14}, and \eqref{shi5.16} we have that $\int_{\rho(t_0)}^{\infty} \|Q(t,\lambda) \|\nabla t < \infty$, so that \eqref{shi5.15} follows. 

We are ready to show that $X(t,\lambda)$ is bounded on $[\rho(t_0),\infty)_\T$. For a solution of \eqref{shi5.6} to exist, we need the coefficient matrix $Q$ to be $\nu-$regressive, in other words we need to show that $I_{2d}-\nu(t)Q(t)$ is invertible for all $t\in[t_0,\infty)_\T$. From \eqref{Atilde}, \eqref{shi5.7}, \eqref{shi5.8}, and \eqref{shi5.9} we have that
$$ I_{2d}-\nu(t)Q(t) = Z^{-1}(t,\lambda)(\Upsilon\Phi)^{*}(t,\lambda_0) \begin{pmatrix} -(\lambda-\lambda_0)\nu(t)W_1(t) & -E^{*-1}(t) \\ 
       E^{-1}(t) & 0 \end{pmatrix} (\Upsilon\Phi)(t,\lambda_0); $$
since $\Phi(\cdot,\lambda_0)$ is a fundamental matrix, by \eqref{shi5.9} again we see that every matrix on the right-hand side here is invertible, making $I_{2d}-\nu(t)Q(t)$ invertible for all $t\in[t_0,\infty)_\T$. Therefore
$$ X(t,\lambda) = \hat{e}_{Q(\cdot,\lambda)}(t,\rho(t_0))X(\rho(t_0),\lambda) \stackrel{\eqref{shi5.4}}{=}\hat{e}_{Q(\cdot,\lambda)}(t,\rho(t_0)) $$
is a well-defined matrix, and thus
$$ \|X(t,\lambda)\|_1 \le \exp\left\{\int_{\rho(t_0)}^{t}\|Q(s,\lambda)\|_1\nabla s\right\}. $$
Combining this with \eqref{shi5.15} we conclude that $\|X(t,\lambda)\|_1$ is bounded for all $t\in[\rho(t_0),\infty)_\T$.

Let us now show that all solutions of \eqref{maineq}$_{\lambda}$ are in $L^2_W\left([\rho(t_0),\infty)_\T\right)$. From the second line of \eqref{shi5.5} and \eqref{shi5.9} we have that
\begin{eqnarray*}
 (\Upsilon\Phi)^*(t,\lambda)W(t)(\Upsilon\Phi)(t,\lambda)
 &=& X^*(t,\lambda)(\Upsilon\Phi)^*\left(t,\lambda_0\right)W(t)(\Upsilon\Phi)\left(t,\lambda_0\right)X(t,\lambda) \\
 && -\nu(t)  X^*(t,\lambda)(\Upsilon\Phi)^*\left(t,\lambda_0\right)\diag\{0,W_2(t)\}(\Upsilon\Phi)\left(t,\lambda_0\right)X^\nabla(t,\lambda) \\
 && -\nu(t) X^{\nabla*}(t,\lambda)(\Upsilon\Phi)^*\left(t,\lambda_0\right)\diag\{0,W_2(t)\}(\Upsilon\Phi)\left(t,\lambda_0\right)X(t,\lambda) \\
 && +(\nu(t))^2 X^{\nabla*}(t,\lambda)(\Upsilon\Phi)^*\left(t,\lambda_0\right)\diag\{0,W_2(t)\}(\Upsilon\Phi)\left(t,\lambda_0\right)X^\nabla(t,\lambda);
\end{eqnarray*}
using the simple formula $\nu X^\nabla=X-X^\rho$ again, we simplify this to
\begin{eqnarray*}
 (\Upsilon\Phi)^*(t,\lambda)W(t)(\Upsilon\Phi)(t,\lambda)
 &=& X^{*}(t,\lambda)(\Upsilon\Phi)^*\left(t,\lambda_0\right)\diag\{W_1(t),0\}(\Upsilon\Phi)\left(t,\lambda_0\right)X(t,\lambda) \\
 &&  +  X^{*\rho}(t,\lambda)(\Upsilon\Phi)^*\left(t,\lambda_0\right)\diag\{0,W_2(t)\}(\Upsilon\Phi)\left(t,\lambda_0\right)X^\rho(t,\lambda).
\end{eqnarray*}
From the boundedness of $\|X(t,\lambda)\|_1$ and \eqref{shi5.16}, we see that
$$ \int_{\rho(t_0)}^{\infty}\|(\Upsilon\Phi)^*(t,\lambda)W(t)(\Upsilon\Phi)(t,\lambda)\|\nabla t < \infty, $$
putting $\Phi(\cdot,\lambda)\in L^2_W\left([\rho(t_0),\infty)_\T\right)$.
\end{proof}


\begin{remark}
Given the above results, we are able to generalize the remainder of \cite[Section 5]{shi2} verbatim, as the proofs are unchanged. For completeness we include the main results here.
\end{remark}


\begin{definition}
Let $d_{\pm}$ be the positive and negative defect indices of $H_0$, where $H_0$ is the minimal operator corresponding to system \eqref{maineq} defined in \eqref{andH04.9}. Then the dynamic Hamiltonian nabla operator $\mathscr{L}$ given in \eqref{shieq15} is said to be in the limit $(d_+,d_-)$ case at $t=\infty$. In the special case of $d_+=d_-=d$, $\mathscr{L}$ is said to be in the limit point case $(l.p.c.)$ at $t=\infty$, and in the other special case of $d_+=d_-=2d$, $\mathscr{L}$ is said to be in the limit circle case $(l.c.c.)$ at $t=\infty$. 
\end{definition}


\begin{remark}
It is clear that there may be at most $1+d^2$ cases for the singular dynamic Hamiltonian nabla system \eqref{maineq} of degree $d$ by the largest defect index theorem (Theorem $\ref{shithm5.5}$) and by using the fact that $d\le d_{\pm}\le 2d$. However, in the special case of $d=1$, the classification is simple just like the formal self-adjoint second-order scalar difference operators; in other words, $\mathscr{L}$ is either in $l.p.c.$ or in $l.c.c.$ at $t=\infty$ by using the largest defect index theorem.
\end{remark}


\begin{theorem}
Assume \eqref{Phinormal}. Then the following nine statements are equivalent.
\begin{enumerate}
 \item $\mathscr{L}$ is in $l.c.c.$ at $t=\infty$;
 \item system \eqref{maineq}$_{\lambda}$ has $2d$ linearly independent solutions in $L^2_W\left([\rho(t_0),\infty)_\T\right)$ for all $\lambda\in\C$ with $\im\lambda\neq 0$;
 \item $\mathscr{R}_0(\lambda)$ is invertible for all $\lambda\in\C$ with $\im\lambda\neq 0$;
 \item the limiting set $\overline{E}_0(\lambda)$ has nonempty interior for all $\lambda\in\C$ with $\im\lambda\neq 0$;
 \item $\lim_{b\rightarrow\infty}\mu_j(b)=\gamma_j$ is finite and positive for $1\le j\le d$, where $\mu_j(b)(1\le j\le d)$ are eigenvalues of $F_{22}(b,\lambda)$ for all $\lambda\in\C$ with $\im\lambda\neq0$;
 \item system \eqref{maineq}$_{\lambda_0}$ has $2d$ linearly independent solutions in $L^2_W\left([\rho(t_0),\infty)_\T\right)$ for some $\lambda_0\in\C$ with $\im\lambda_0\neq 0$;  
 \item $\mathscr{R}_0(\lambda_0)$ is invertible for some $\lambda_0\in\C$ with $\im\lambda_0\neq 0$;
 \item the limiting set $\overline{E}_0(\lambda_0)$ has nonempty interior for some $\lambda_0\in\C$ with $\im\lambda_0\neq 0$;
 \item $\lim_{b\rightarrow\infty}\mu_j(b)=\gamma_j$ is finite and positive for $1\le j\le d$, where $\mu_j(b)(1\le j\le d)$ are eigenvalues of $F_{22}(b,\lambda_0)$ for some $\lambda_0\in\C$ with $\im\lambda_0\neq 0$.
\end{enumerate}
\end{theorem}

Similarly, the following equivalent conditions on the limit point case can be concluded by Theorems \ref{shithm4.3} and \ref{shithm5.3}.


\begin{theorem}
Assume \eqref{Phinormal}. Then the following seven statements are equivalent:
\begin{enumerate}
 \item $\mathscr{L}$ is in $l.p.c.$ at $t=\infty$;
 \item system \eqref{maineq}$_{\lambda}$ has $d$ linearly independent solutions in $L^2_W\left([\rho(t_0),\infty)_\T\right)$ for all $\lambda\in\C$ with $\im\lambda\neq 0$;
 \item $\mathscr{R}_0(\lambda)=0$ for all $\lambda\in\C$ with $\im\lambda\neq 0$;
 \item $\lim_{b\rightarrow\infty}\mu_1(b)=\infty$, where $\mu_1(b)$ is the smallest eigenvalue of $F_{22}(b,\lambda)$ for all $\lambda\in\C$ with $\im\lambda\neq0$;
 \item  systems \eqref{maineq}$_{\lambda_0}$ and \eqref{maineq}$_{\overline{\lambda}_0}$ have exactly $d$ linearly independent solutions in $L^2_W\left([\rho(t_0),\infty)_\T\right)$, respectively, for some $\lambda_0\in\C$ with $\im\lambda_0\neq 0$;
 \item $\mathscr{R}_0(\lambda_0)=\mathscr{R}_0\left(\overline{\lambda}_0\right)=0$ for some $\lambda_0\in\C$ with $\im\lambda_0\neq 0$;
 \item $\lim_{b\rightarrow\infty}\mu_1(b)=\infty$, where $\mu_1(b)$ is the smallest eigenvalue of $F_{22}(b,\lambda_0)$ and the smallest
eigenvalue of $F_{22}\left(b,\overline{\lambda}_0\right)$ for some $\lambda_0\in\C$ with $\im\lambda_0\neq 0$.
\end{enumerate}
\end{theorem}

If all the coefficients of system \eqref{maineq}$_{\lambda}$ are real, we have the following results by Theorems $\ref{shi5.4}$ and $\ref{shi5.7}$.


\begin{cor}\label{shicor5.1}
If $P(t)$ and $W(t)$ are real for all $t\in[\rho(t_0),\infty)_\T$, then following nine statements are equivalent.
\begin{enumerate}
 \item $\mathscr{L}$ is in $l.p.c.$ at $t=\infty$;
 \item system \eqref{maineq}$_{\lambda}$ has exactly $d$ linearly independent solutions in $L^2_W\left([\rho(t_0),\infty)_\T\right)$ for all $\lambda\in\C$ with $\im\lambda\neq 0$;
 \item $\mathscr{R}_0(\lambda)=0$ for all $\lambda\in\C$ with $\im\lambda\neq 0$;
 \item the limiting set $E_0(\lambda)$ contains only one element for all $\lambda\in\C$ with $\im\lambda\neq 0$;
 \item $\lim_{b\rightarrow\infty}\mu_1(b)=\infty$, where $\mu_1(b)$ is the smallest eigenvalue of $F_{22}(b,\lambda)$ for all $\lambda\in\C$ with $\im\lambda > 0$;
 \item system \eqref{maineq}$_{\lambda_0}$ has exactly $d$ linearly independent solutions in $L^2_W\left([\rho(t_0),\infty)_\T\right)$ for some $\lambda_0\in\C$ with $\im\lambda_0\neq 0$;  
 \item $\mathscr{R}_0(\lambda_0)=0$ for some $\lambda_0\in\C$ with $\im\lambda_0\neq 0$;
 \item the limiting set $E_0(\lambda_0)$ contains only one element for some $\lambda_0\in\C$ with $\im\lambda_0\neq 0$;
 \item $\lim_{b\rightarrow\infty}\mu_1(b)=\infty$, where $\mu_1(b)$ is the smallest eigenvalue of $F_{22}(b,\lambda_0)$ for some $\lambda_0\in\C$ with $\im\lambda_0\neq 0$.
\end{enumerate}
\end{cor}

As a consequence of Theorems \ref{shithm4.4}, \ref{shithm5.3}, and \ref{shithm5.5}, Corollary \ref{shicor5.1} holds in the special case of $d=1$, no matter if the coefficients of \eqref{maineq}$_{\lambda}$ are real or complex.


\begin{cor}
If $d=1$, then the equivalent statements $\rm{(i)}-\rm{(ix)}$ in Corollary $\ref{shicor5.1}$ hold.
\end{cor}


\begin{remark}
Much of the theory of Weyl and Titchmarsh remains that can be extended to time scales, such as $M(\lambda)$ theory in the limit point case \cite[Section 6]{shi2}, asymptotic expansion of Weyl-Titchmarsh matrices and Green's matrices \cite{cg2}, and so on, leaving the future of the subject open to interested researchers.
\end{remark}


\section{alternative form}

In this section we introduce a possible alternative form for this theory on Sturmian time scales. For example, instead of system \eqref{maineq}, consider the alternative system
\begin{equation}\label{alteq}
 J(\Upsilon y)^\Delta(t)=\Big(\lambda W(t)+P(t)\Big) y(t), \quad t\in[t_0,\infty)_\T, \quad J=\left(\begin{smallmatrix} 0_n & -I_n \\ I_n & 0_n \end{smallmatrix}\right),
\end{equation}
for the same block matrices $W$ and $P$, where we have the delta derivative and $\Upsilon y$ on the left-hand side for $\Upsilon$ in \eqref{leftshift}, $y$ on the right-hand side, and where this time we assume 
\begin{equation}\label{msA1}
  E_2(t):=\Big(I_n+\mu(t) A^*(t)\Big)^{-1} 
\end{equation} 
exists instead of \eqref{Atilde}. System \eqref{alteq} may also be viewed as a generalization of \eqref{shi1.5} and \eqref{ms1.2}. For \eqref{alteq} we have the integration by parts formula (compare with Theorem \ref{Lagrange})
$$ \int_{a}^{b} \left[z^*J(\Upsilon y)^\Delta - \left(J(\Upsilon z)^\Delta\right)^*y \right](t)\Delta t 
         = (\Upsilon z)^*(b)J(\Upsilon y)(b) - (\Upsilon z)^*(a)J(\Upsilon y)(a). $$
Moreover, the scalar product in \eqref{ipinfty} is now replaced by a standard weighted product without shifts given by
\begin{equation}\label{altprod}
 (x,y)_W:=\int_{t_0}^{\infty}y^*(t)W(t)x(t)\Delta t, \quad x,y\in L_W^2([t_0,\infty)_\T). 
\end{equation}
Additionally, as in \eqref{nabsym} we may rewrite \eqref{alteq} as the system 
\begin{equation}\label{delsym}
 (\Upsilon y)^\Delta(t)=\mathcal{K}(t,\lambda)(\Upsilon y)(t), \quad \mathcal{K}(\cdot,\lambda):=-J(\lambda W+P)\widehat{H},
\end{equation}
where on $[t_0,\infty)_\T$ we use $E_2=(I_n+\mu A^*)^{-1}$ and
\begin{equation}\label{ms2.2}
 \widehat{H}:=\begin{pmatrix} I_n & 0_n \\ \mu E_2(C-\lambda W_1) & E_2 \end{pmatrix}, \quad\text{with}\quad -J(\lambda W+P)=\begin{pmatrix} A & \lambda W_2+B \\ C-\lambda W_1 & -A^* \end{pmatrix},
\end{equation}
since $E_2A^*=A^*E_2$ and $I-\mu A^*E_2=E_2$. 
Directly from the definition of $\mathcal{K}(\cdot,\lambda)$ in \eqref{delsym} we have that
$$ I_{2n}+\mu(t)\mathcal{K}(t,\lambda) = \begin{pmatrix} I_n+\mu(t)A(t) & \mu(t)(\lambda W_2(t)+B(t)) \\ 0_n & I_n \end{pmatrix} \widehat{H}(t), $$
so that $I_{2n}+\mu\mathcal{K}(\cdot,\lambda)$ is invertible by \eqref{msA1}, $\mathcal{K}(\cdot,\lambda)$ is regressive, and the matrix equation $(\Upsilon y)^\Delta=\mathcal{K}(\cdot,\lambda)\Upsilon y$ with initial condition $(\Upsilon y)(t_0)=y_0$ has a unique solution $\Upsilon y$ on $[t_0,\infty)_\T$. It follows that an initial value problem involving \eqref{alteq} has a unique solution $y$ in
$$ \left\{y=(y_1,y_2)^{\trans}\Big|\; y_1,y_2^\rho:[t_0,\infty)_\T\rightarrow\C^n \; \text{are delta differentiable} \right\}. $$

In summary, to unify \eqref{shi1.5} and \eqref{ms1.2} on Sturmian time scales, systems equivalent to \eqref{maineq} or \eqref{alteq} must be used to account for the shifts in the discrete case \cite{shi2}. For such systems to admit the existence and uniqueness of solutions to initial value problems, an integration by parts formula, and a matrix weighted scalar product in a Hilbert space, the Sturmian assumption \eqref{sturmfact} is essential.


\end{document}